\documentclass[11pt]{amsart}
\usepackage{mathtext}         %

\usepackage{latexsym}
\usepackage{amsmath,amssymb,amsthm,amsopn,amscd}
\usepackage{euscript}

\usepackage[dvips]{graphicx}
\usepackage{array,fullpage,wrapfig}

\usepackage[matrix,arrow,curve]{xy}
\usepackage[active]{srcltx}

\newcommand{\picref}[1]{\ref{#1}}

\numberwithin{equation}{subsection}

\theoremstyle{plain}
\newtheorem{theorem}[equation]{Theorem}

\newtheorem{proposition}[equation]{Proposition}
\newtheorem{thesis}{Expectation}
\newtheorem{prop}[equation]{Proposition}
\newtheorem{lemma}[equation]{Lemma}
\newtheorem{cor}[equation]{Corollary}

\newtheorem{conj}[equation]{Conjecture}
\newtheorem*{proposition*}{Proposition}
\newtheorem*{theorem*}{Theorem}
\newtheorem*{thesis*}{Expectation}

\newtheorem{corollary}[equation]{Corollary}

\newtheorem{rema}[equation]{Remark}

\theoremstyle{definition}

\newtheorem{definition}[equation]{Definition}

\newtheorem{exam}[equation]{Example}

\newtheorem{proposition-definition}[equation]{Proposition--Definition}

\newtheorem{example}[equation]{Example}

\newcommand{\QQ}{\ensuremath{\mathbb Q}}

\newcommand{\Z}{\ensuremath{\mathbb Z}}

\newcommand\alia{ {\mathcal A}lia } 
\newcommand \B{\mathsf{B}} 
\renewcommand \b{\mathcal{B}} 

\renewcommand\dim{\mbox{dim\,}}
\newcommand \F{\EuScript{F}}
\newcommand \f{\mathcal{F}}
\renewcommand \H{\EuScript{H}}   

\newcommand\Id{\mbox{Id\,}}

\newcommand \M{\EuScript{M}} 
\newcommand \m{\mathcal{M}}  
\newcommand\NU{\mathcal {NU}}
\newcommand\p{ {\mathcal P} } 
\renewcommand\P{\EuScript{P}} 
\newcommand \Q{\EuScript{Q}}  
\newcommand \q{\mathcal{Q}}   

\newcommand{\sh}{\ensuremath{{Sh}}}

\newcommand\ass{{\mathcal{A}ssoc}}
\newcommand\com{\: {\mathcal C}om \:}

\newcommand\lie{\: {\mathcal L}ie \:}

\newcommand\kk{\Bbbk}

\let\le\leqslant
\let\ge\geqslant
\let\leq\leqslant
\let\geq\geqslant

\newcommand \GKdim{\mathsf{Dim}}
\newcommand \Stump{\mathsf{Stump}}

\begin{document}


\title{On generating series of finitely presented operads}%

\author[A.~Khoroshkin]{Anton Khoroshkin}
\address{
Laboratory of Mathematical Physics \&
Faculty of mathematics,  Vavilova 7, 
National Research University Higher School of Economics,
 Moscow 117312, Russia \\
and Institute of Theoretical and Experimental Physics (ITEP Moscow),
B. Cheremushkinskaya, 25, Moscow, 117259, Russia }
\email{akhoroshkin@hse.ru}

\author{Dmitri Piontkovski}
 \address{School of Mathematics,  Faculty of Economics,
Myasnitskaya str. 20, 
National Research University Higher School of Economics,
 Moscow 101990, Russia
}
\email{piont@mccme.ru}

\thanks{
The first author's research is partially supported by RFBR grants 13-02-00478, 13-01-12401, 15-01-09242,  
by "The National Research University--Higher School of Economics" Academic Fund Program in 2013-2014, research grant 14-01-0124, 
by Dynasty foundation and Simons-IUM fellowship.
The second author's research was supported by ``The National Research University Higher School of Economics Academic Fund Program'' in 2013--2014,
 research grant  12-01-0134, and the RFBR project 
 14-01-00416.}

 \begin{abstract}
 Given an operad $P$  with a finite Gr\"obner basis
 of relations,  we study the generating functions for the dimensions of its graded
 components $P(n)$.
 Under moderate assumptions on the relations
 we prove that the exponential generating function for the sequence $\{ \dim P(n) \}$
 is differential algebraic, and in fact algebraic if $P$ is a symmetrization of a non-symmetric operad.
 If, in addition, the  growth of the dimensions of  $ P(n)$ is bounded by an exponent of $n$
 (or a polynomial of $n$, in the non-symmetric case) then, moreover, the
ordinary generating function
 for  the above sequence $\{ \dim P(n) \}$ is rational.
 We give a number of examples of calculations and discuss conjectures about the above generating functions for more general
 classes of operads.
  \end{abstract}

\maketitle

\setcounter{section}{-1}


\section{Introduction}

We study the generating series for the dimensions of the components of algebraic operads.
For symmetric operads, we conjecture 
 that under mild restrictions these exponentional generating series are differential algebraic.
We prove that this is indeed the case if the operad has a finite Gr\"obner basis which satisfies  an additional condition. Moreover, we show that if the dimensions of the components of the operad are bounded by an exponential function, then the corresponding generating function is rational. For non-symmetric operads, we show that the ordinary generating series is  algebraic if the operad has a finite  Gr\"obner basis. Moreover, 
the series is a rational function if, in addition, the dimensions of the operad components  are bounded 
by a polynomial function. 
We also describe several algorithms for calculating the above series in various  situations,
 and provide a number of examples of calculations.
In particular, there are several
natural examples of operads for
which the generating series were not previously known.

\subsection{Main results}

Let $\p$ be a finitely generated operad over a field $\kk$ of characteristic zero.
Recall that an
exponentional generating series of $\p$ is defined as
\begin{equation}
\label{eq::E::gen::ser}
E_{\p} (z) := \sum_{n \ge 1} \frac{\dim \p(n)}{n!} z^n .
\end{equation}
We also consider the ordinary
generating series
\begin{equation}
\label{eq::G::gen::ser}
G_{\p} (z) := \sum_{n \ge 1} {\dim \p(n)} z^n .
\end{equation}
In particular,
if $\p$ is  a symmetrization of a non-symmetric operad $\P$, then $\dim\p(n) = n!\dim\P(n)$ so that 
$E_{\p} (z) = G_\P (z)$.

These generating series are important invariants of an operad. 
For example, they appear  in the Ginzburg--Kapranov criterion for Koszul operads. 
\cite{zin} gives examples of  such series. 
Moreover, the generating series for binary
symmetric operads with one or two generators  have been extensively studied for decades under the name of
{\em codimension series of varieties of algebras},
see~\cite{giza, bd}. These are essentially the series for quotients of the operad of associative algebras.

Here we present a new approach to such series
using the theory of Gr\"obner bases for operads developed
 in~\cite{dk}. Note that, in general, the theory of such Gr\"obner bases helps to answer a question: is a given operadic element equal to zero in an operad defined by a given collection of operadic relations? For example, the famous Jacobian Conjecture can be reformulated as a question of this kind~\cite[Section~2]{onYagzhev}.
A number of  operads admit a finite Gr\"obner basis of the ideal of relations; we refer them as operads with a finite Gr\"obner basis for short. In such operads, there is an effective direct algorithm to answer the above question.  We study  the generating series for the dimensions of the components of such operads.

Suppose an operad $\P$ admits a finite Gr\"obner basis $G$ of the ideal of its relations.
Let $\widehat G$ be the set of leading terms of the elements of $G$.
The operad $\widehat \P$ defined by  the same set of generators as $\P$ and the finite set of relations $\widehat G$ 
is called the monomial replacement of $\P$ and has the same generating series of dimensions as $\P$.
This observation reduces the problem of the describing the generating series of operads with finite Gr\"obner bases
to the same problem for operads with finite set of monomial relations. 
Thus, we prove the theorems on operads with finite Gr\"obner bases listed below for operads 
with a finite set  of monomial relations only. 

Our first result deals with {\em non-symmetric} operads.

\begin{theorem}[Theorem~\ref{th-nonsym}]
\label{th-nonsym-intro}
 The ordinary generating series of a non-symmetric operad with a finite Gr\"obner
 basis is an algebraic function.
\end{theorem}

In the more general symmetric case,
the analogous result is true under the additional assumptions explained below.

By definition~\cite{dk}, the leading terms of the elements of a reduced Gr\"obner
basis $G$ of a symmetric operad are {\em shuffle monomials}, that is,
rooted trees with an additional labelling of vertices and leaves.
Each shuffle monomial
has a unique planar  representative (see details in
Section~\ref{sec::sym_background}), so we can identify a shuffle monomial
with a planar tree whose leaves are enumerated by an  initial segment of positive integers
permuted by a so-called {\em shuffle substitution}. We  call a set $M$ of
shuffle monomials {\em shuffle regular}
(Definition~\ref{def:shuffle_regular}) if for each shuffle
monomial $m$ in $M$
and for each shuffle substitution $\sigma$ of
its leaves, the monomial $m'$ obtained from $m$ by acting with $\sigma$
on its leaves also belongs to $M$. An operad with a given
Gr\"obner basis $G$  is called {\em shuffle regular} if the
set  $\widehat G$  
of the leading monomials of the elements of the Gr\"obner basis
is shuffle regular.

Recall  that a
 function or
a formal power series
is called {\em differential
algebraic} if it satisfies a non-trivial  algebraic differential
equation with polynomial coefficients.

\begin{theorem}[Corollary~\ref{cor:lin_syst}]
\label{th:main_sym_intro} Let $\p$ be a shuffle regular symmetric
operad with a finite Gr\"obner basis. Then its
generating series $E_{\p}$ is differential algebraic.
\end{theorem}

We also consider two special classes of operads which give rise to a generating series of a more special form.
The first class consists of monomial shuffle regular operads 
whose relations satisfy an additional symmetry condition.
Namely, {\em if the set of relations of a monomial operad $\p$ forms a set of all  planar representatives of a given set of non-planar trees,
then the generating series $E_{\p}$ is algebraic} (Theorem~\ref{th::bosquet}).
The second class is defined by the following bounds for the dimension growth
of  the components of the operad.

\begin{theorem}[Corollaries~\ref{cor::nonsym_constant_growth} and~\ref{cor::sym_constant_growth}]
\label{th:intro_slow_growth}
Let $\p$ be an operad with a finite Gr\"obner basis.
Suppose that either \\
$\phantom{1111}(i)$ $\p$ is non-symmetric and the numbers $\dim \p(n)$ are bounded by some polynomial
in $n$\\
or\\
$\phantom{1111}(ii)$ $\p$ is shuffle regular and the dimensions $\dim \p(n)$  are bounded by an exponential function
$a^n$ for some $a>1$.\\
Then  the generating series $G_\p$ is rational.
\end{theorem}

In fact, the  growth conditions we need in Theorem~\ref{th:intro_slow_growth} are even weaker, 
see Corollaries~\ref{cor::nonsym_constant_growth} and~\ref{cor::sym_constant_growth}.

All our proofs are constructive and provide a number of methods of obtaining  the corresponding
algebraic or differential algebraic equations. Depending on the situation  one of these may be more effective.
They are based on relations between right sided ideals in monomial operads, 
on homological computations, and on using symmetries of the relations of operads. 
Each method is illustrated by a number of examples.

\begin{rema}
Sometimes there is an additional integer grading on the operad $\p$ such that all vector spaces $\p_n$ are graded, 
and such that the grading is additive with respect to the compositions of homogeneous elements.
Then one can also consider the two-variable generating functions $G_{\p} (z,t)$ and $E_{\p} (z,t)$,
 see Sections~\ref{sec::nonsym_grading} and~\ref{sec::sym_add_grading}.
Theorems~\ref{th:main_sym_intro},~\ref{th-nonsym-intro},~\ref{th:intro_slow_growth} 
remain valid after replacing the coefficient  field $\QQ$
in the differential and algebraic equations  by the ring $\QQ[t]$.
\end{rema}

\subsection{Outline of the paper}

In Section~\ref{sec::recall}, we give a brief introduction to algebraic operads and
  Gr\"obner bases.
Subsections~\ref{sec::nonsym_background} and ~\ref{sec::sym_background}
contain a detailed descriptions of monomials in free non-symmetric and symmetric operads respectively.

 In Section~\ref{sec::nonsym}, we illustrate our ideas in the simpler case
  of non-symmetric operads. 
In Subsection~\ref{sec::many_relations}, we begin to prove Theorem~\ref{th-nonsym-intro} 
by  finding an algebraic relation for the generating series of an arbitrary non-symmetric operad
with a finite Gr\"obner basis.
Our proof is based on considering the principal ideals with bounded degree  of generators
in a finitely presented monomial operad. We provide an algorithm to construct a system of polynomial relations 
of the form $y_i = F_i(y_1, \dots, y_N)$ 
for the generating functions $y_1, \dots, y_N$ of these ideals.
In Subsections~\ref{sec::few_relations} and~\ref{sec::homology_algorithm},
we give two other versions of this algorithm optimized for  few relations 
and for a simple structure of relations, respectively.
Since the generating series of the operad itself is a linear combination of some $y_i$,
 we deduce in Subsection~\ref{sec:alg_eq_nonsym} that the latter generating series satisfy an algebraic equation. 
Then we discuss a  bound for the degree of this equation.
In Subsection~\ref{sec::nonsum_const_growth}, we investigate the case of operads with subexponential growth of the dimensions 
of its components and prove the first part of Theorem~\ref{th:intro_slow_growth}.

In Section~\ref{sec::sym_operads}, we deal with symmetric operads.
In Subsection~\ref{sec::operator_C}, we prove a formula for the generating series of the shuffle composition of sequences of linear subspaces.
 This is the first place where integral and differential equations appear.
In Subsection~\ref{sec::sym_main_theorem}, we prove Theorem~\ref{th:main_sym_intro}.
In Subsection~\ref{sec::sym_bosquets},
we consider
 a  class of shuffle regular operads with additional symmetries of the relations.
For such an operad (called {\em symmetric regular}), the set of leading monomials of the Gr\"obner basis
 is, by definition, closed
 under an arbitrary re-numbering of the leaves.
We show that the exponential generating series of the symmetric regular operads are algebraic functions.
In Subsection~\ref{sec::sym_constant_growth},
we consider shuffle regular operads with slow growth of the dimensions of the components and prove the shuffle part of
Theorem~\ref{th:intro_slow_growth}.
The  examples of symmetric operads are collected in Subsection~\ref{sec::examples:sym}.

In Section~\ref{sec::final_remarks} there are additional remarks and conjectures.
In Subsection~\ref{subs:motiv}, we discuss an analogy of operadic generating functions and the Hilbert series of associative graded algebras. 
We observe that many common algebras have rational Hilbert series and formulate three informal 
conjectures on the generating functions of common finitely presented operads.
These expectations generalize our theorems. 
In Subsection~\ref{subs:conj_PBW}, we give some evidence for the conjecture that each binary operad with a quadratic Gr\"obner basis
has the same generating series as some shuffle regular operad (hence, this series is differential algebraic).
In Subsection~\ref{subs:conj_ass}, we discuss 
some evidence showing that  operads of associative algebras with polynomial identities are ``generic''. After this paper was sibmitted, Berele
has proved that these operads indeed have holonomic (hence, differential algebraic) generating series~\cite{berele}. 
 Third, we remark that the operation of shuffle composition induces the structure of Zinbiel algebra on the ring
of formal power series. Therefore, it is interesting to describe the minimal Zinbiel subalgebras which contain the generating series of some classes of operads, and to describe the class of ``Zinbiel algebraic'' functions which contains these generating series.

\subsection{Considered examples.}
We illustrate our methods by a number of examples. For the reader's convenience, 
we give below a brief list of examples with references to the text.
In most of these examples, the operads are generated by a single binary operation (multiplication).
As usual, we denote $[a,b] := ab-ba$, $(a,b,c):=(ab)c-a(bc)$ and $[a,b]_{+}:=ab+ba$.
These examples are:

\begin{itemize}

 \item
 a non-symmetric operad with the identities $((a b) c) d) =0 \text{ and }  (a(b((cd)e))) = 0$
(Example~\ref{eq::Asw::rec:rel});

 \item
  non-symmetric operad $\Q_k$ with generalized associativity relation
\\
$
x_1(\dots ( x_{k-2}(x_{k-1}, x_k, x_{k+1})) \dots )=0
$ (Example~\ref{ex::nonsym_homol_exam});

 \item
(Example~\ref{ex:alia})
 operads of alia algebras, left alia algebras and right alia algebras~\cite{dzh},
that is, the operads generated by a single binary nonsymmetric operation defined by the identity
$$
\begin{array}{cc}
\text{ (alia algebras) } & [[x_1, x_2], x_3]_{+} + [[x_2, x_3], x_1]_{+} + [[x_3, x_1], x_2]_{+} = 0, \\
\text{ (left alia algebras) } & [x_1, x_2]x_3 + [x_2, x_3]x_1 + [x_3, x_1]x_2 = 0,\\
\text{ (right alia algebras) } &
x_1[x_2, x_3] + x_2[x_3, x_1] + x_3[x_1, x_2] = 0.
\end{array}
$$

 \item
 operad ${\mathcal {NU}}_n$ of upper triangular matrices of order $n$
over a nonassociative commutative ring (Example~\ref{ex::2x2-upper-matrices} for $n=2$,
Example~\ref{exam::NU3} for $n=3$ and Lemma~\ref{lem:upper-triangle-define} for the general case).

 \item
 a Lie-admissible  operad  defined by the set of identities $\{g([-,-],\ldots,[-,-])\}$
 where the skew-symmetric generator $[-,-]$
satisfies the Jacobi identity and $g$ is a collection of linear combinations of compositions of other generators.
(Example~\ref{ex::sym_homol_exam});

 \item
 an operad with the identities $
[x,[y,z]]+[y,[z,x]]+[z,[x,y]] = 0
$
and
$
[x,y][z,t] + [z,t][x,y] =0
$
(Example~\ref{ex::hom_exam_3}, as a special case of the previous example).
\end{itemize}

\subsection*{Acknowledgments}

We are grateful to Vladimir Dotsenko, Askar Dzhumadil'daev, Samuel Grushevsky and Boris Shapiro for stimulating discussions
and comments on the exposition.
An essential part of the paper was written while the first author was working at the ETH and
he wishes to thank  all the staff of ETH and especially
Giovanni Felder who in addition helped to organize a short visit of the second author to ETH.

\section{Operads and trees}

In the brief and rough explanation below, we refer the reader to the
books~\cite{Loday_Valet, oper} for the details on the basic facts on operads and to~\cite{dk} and~\cite[Ch.~8]{Loday_Valet}  for the details
on shuffle operads and Gr\"obner bases in operads.

\label{sec::recall}

\subsection{Operads in few words}

Roughly speaking, a (linear) operad is a way to define a type, or a variety, of linear algebraic systems.  
A linear operad $\P$ consists of a collection (or a disjoint union)
of vector spaces $\P(n)$  of all multilinear operations
(numbered by the amounts of inputs) and multilinear 
operations $\circ_i : \P(n)\otimes \P(m) \to \P(n+m-1)$ called compositions\footnote{Note that the compositions  $\circ_i$ are referred as {\em partial compositions} in~\cite{Loday_Valet}.
The term {\em composition} here has been used by Gerstenhaber, see~\cite[I.1.4]{oper}.} which prescribe how to substitute
into the $i$-th input of an operation from $\P(n)$  the result of an operation from $\P(m)$.
The details are discussed below.
All operads we consider are linear. 
The set $\P(n)$ is identified with all possible $n$-linear operation of the algebraic system. The algebraic system itself is then called an algebra over the operad $\P$.

One typically separates two versions of this notion
 (the so-called symmetric and non-symmetric operads)
depending on whether we allow and do not allow to 
permute the inputs of operations.
The permutations  of inputs define the action of symmetric groups on operations and compositions,
hence, the direct computations becomes much harder in the symmetric case.
On the other hand, the symmetric operads are much more important for applications.
The general theory of non-symmetric operads is similar but more transparent,
since one needs not take care of
the symmetric group action.

Given a discrete set $\Upsilon = \cup_{n\geq 1} \Upsilon_n$ of (abstract) multilinear operations, where
$ \Upsilon_n$  consists of the $n$-ary operations, one can define the {\em free}  operad
$\F = \F(\Upsilon) = \cup_{n\geq 1} {\F}(n)$ generated by $\Upsilon$.
 Here ${\F}(n)$ is the vector space spanned by
all $n$-ary operations which are {\em compositions} of the elements of $\Upsilon$
 with each other.
The composition of the empty set of operations is assumed to be a unary identity operation $Id\in\F(1)$.
Note that in the symmetric case
a `composition' may permute the inputs, so that the set of $n$-ary generators $\Upsilon_n$ should span a representation of the symmetric group.
 To help the reader to separate the non-symmetric and symmetric cases, we
will usually denote the corresponding free operads by $\F(\Upsilon)$ and $\f(\Upsilon)$, using script and calligraphic fonts respectively.

Each operad can be defined in terms of generators and relations.
Therefore,
each operad may be considered as a quotient of a free operad.
That is why we first describe linear bases in the free operads
(referred to as {\em monomial bases})
which are compatible with compositions of operations.
The difference between the symmetric and the non-symmetric cases is essential already on this level.
Namely, the definition of  monomial basis for non-symmetric operads is natural whereas
to define monomials for symmetric operads we need an additional structure of shuffle operad introduced in~\cite{dk}.
We discuss the definitions and combinatorics of non-symmetric and shuffle monomials in the forthcoming subsections.

\subsection{A basis for a free non-symmetric operad: tree monomials}
\label{sec::nonsym_background}
Let us recall the combinatorics of trees involved in the description of monomials in operads.

Basis elements of the free operad are represented by decorated trees.
By a (rooted) \emph{tree} we mean a non-empty connected oriented graph $T$ of genus zero for which
each vertex has at least one incoming edge and exactly one outgoing edge. Some edges  of a tree  have 
a vertex at one end only. Such edges are called \emph{external}. All other edges (having
vertices at both ends) are called \emph{internal}. Each tree has one outgoing external edge
(= the \emph{output} or the \emph{root}) and several ingoing external edges, called \emph{leaves} or \emph{inputs}.
The number of leaves of a tree $v$ is called the \emph{arity} of the tree and is denoted by $ar(v)$.

Consider the set of generators $\Upsilon = \cup_{n\geq 1} \Upsilon_n$
of the free operad $\F(\Upsilon)$.
We mark a tree with a single vertex and $n$ incoming edges by an arbitrary element of $\Upsilon_n$
and we call such tree a \emph{corolla}.
We say that a rooted tree $T$ \emph{is internally labeled} by the set $\Upsilon$
if each vertex is labeled by an element from $\Upsilon$
such that the number of inputs in this vertex coincides with the arity of the corresponding operation.

A rooted tree is called \emph{planar} if the ordering of incoming edges 
in each vertex is chosen.
(Therefore, there is a canonical way to project a tree on a plane.)

\begin{proposition-definition}\label{lem::basis_in_nonsym}
 A canonical basis of a free non-symmetric operad $\F(\Upsilon)$ is enumerated by the set $\B(\Upsilon)$
of all planar rooted trees
internally labeled by the elements of the set $\Upsilon$.
We shall refer to the elements of this basis as \emph{planar tree monomials}.
\end{proposition-definition}

\begin{wrapfigure}{r}{2.2in}
\caption{A planar tree monomial}
\includegraphics{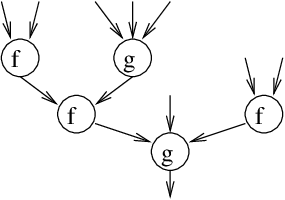}
\label{pic::nonsym::tree}
\end{wrapfigure}

There are two standard ways to think about elements of an operad in terms of its generators.
The first way is in terms of tree monomials  represented by planar trees and the second one is in terms of compositions of operations presented by formulas with brackets.
Our approach is somewhere in the middle: in most cases, we prefer  to think about tree monomials, but to write formulas required for definitions and proofs in the language of operations since it makes things more compact.

A particular example of a planar tree monomial is presented in Figure~\picref{pic::nonsym::tree}.
The corresponding operation may be written as follows:
$g(f(f(\textrm{-},\textrm{-}),g(\textrm{-},\textrm{-},\textrm{-})),\textrm{-},f(\textrm{-},\textrm{-}))$.
It is also convenient to use variables and letters for the inputs of operation.
In a non-symmetric operad, we  enumerate the inputs from the left to the right, 
so that the input variables are (from the left to the right) $x_1,x_2,x_3,$\dots.
For example,
the operation from Figure~\picref{pic::nonsym::tree} will be written in the following form:
$$g(f(f(x_1,x_2),g(x_3,x_4,x_5)),x_6,f(x_7,x_8)).
$$

The composition rules for tree-monomials are given by the concatenation of trees.
The notion of divisibility is defined in the following way:
A monomial $v$ is \emph{divisible} by $w$ if there exists a planar subtree of $v$ isomorphic to $w$
as a planar labeled tree. (Notice that the root of $w$  need not to coincide with the root of $v$.)
For example, the tree-monomial that represents the operation $f(\textrm{-},g(\textrm{-},\textrm{-},\textrm{-}))$ is a divisor
of the tree-monomial given in Figure~\picref{pic::nonsym::tree}; the corresponding planar subtree which contains
the internal edge going from
the ternary operation $g$ in the upper level to the binary operation $f$ in the intermediate level.

Let us specialize two cases of divisibility. A monomial $w$ of some arity $m$ is a {\em left}
divisor of $v$ if $v = w(v_1, \dots, v_m)$ for some monomials $v_1, \dots, v_m$.
By the other words, $w$ is a left divisor of $v$ if there is a
 planar subtree  $w'$ in $v$ which is isomorphic to $w$ and has the same root as $v$.
 Similarly, a monomial $u$ of some arity $l$ is a {\em right} divisor of $v$ if $v = w(x_1, \dots,u(x_i, \dots, x_{i+l-1}), x_{i+l}, \dots, x_{{l+m-1}})$ (where $u$ is in $i$th place) for some monomial $w$ of some arity $m$.
 By the other words, $u$  is a { right} divisor of $v$ if
 all leaves of the subtree $u'$ in $v$ isomorphic to $u$ are also leaves of $v$.
 For example, the empty monomial and the monomial $v$ itself are both left and right divisors of $v$.

\subsection{A basis in  free symmetric operad: shuffle monomials}

\label{sec::sym_background}

A free symmetric operad may  have no monomial basis closed under with all possible compositions.
The way to avoid this problem is to construct a basis closed under  {\em some} compositions.
A class of compositions we would like to preserve is called {\em shuffle compositions}, see below for the definition.
A collection of all multilinear operations with the prescribed rules for shuffle compositions
form a \emph{shuffle} operad~(\cite{dk}; see also~\cite[8.2]{Loday_Valet}).
Any symmetric operad $\p$ may be considered as a shuffle operad denoted by $\p^f$.
Reversely, one can recover from $\p^f$ quite a lot enumerative data associated with $\p$.
In particular, the generating series of the operads $\p$ and $\p^{f}$ are the same.
Since the  symmetric operads are more important for applications than the shuffle ones, 
all examples of shuffle operads discussed below are examples of symmetric operads considered as shuffle operads.

Let us describe a basis
in the free symmetric operad $\f(\Upsilon)$ compatible with the shuffle compositions.
The elements of the basis will be called {\em shuffle monomials}.\footnote{Note that the term {\em tree monomial} has been used in~\cite{dk} for the elements of this basis. 
We use the notion {\em shuffle monomial} to separate the case of a shuffle operad from the case of a non-symmetric one.}
The main difference from the non-symmetric world is that one should have a labeling of the external edges
(i.e., a shuffle monomial is a non-symmetric monomial with a particular ordering of inputs).
We say that a rooted tree with $n$ incoming edges (=leaves) has an \emph{external} labeling
if the set of all leaves is numbered by distinct natural numbers from $1$ to $n$.

\begin{proposition-definition}\label{lem::basis_in_sym}
 The basis of the free shuffle operad $\f(\Upsilon)$ generated by the set $\Upsilon$ is numbered by the set $\b(\Upsilon)$
of all rooted (nonplanar) trees internally labeled by the set $\Upsilon$ and having external labeling.
We shall refer to the elements of this basis as \emph{shuffle tree monomials} or simply \emph{shuffle  monomials}.
\end{proposition-definition}

More generally, we use a term~{\em shuffle tree} for a rooted tree with $n$ incoming edges (leaves) if the leaves are numbered by some distinct natural numbers (not necessary the numbers $1, \dots, n$).

Note that whereas the shuffle tree monomials and shuffle trees  are originally considered as abstract graphs,
 we will use their particular planar representatives  described below.

In general, an embedding of a rooted tree in the plane is determined by an ordering of inputs for each vertex. We assume that the inputs of each vertex are arranged from left to right. 
To compare two inputs of a vertex~$v$, we find for each of these inputs the minimal leaf  that one can reach from~$v$ via it. 
The input for which the minimal leaf is smaller is considered to be less than the other one.
Then assume that a shuffle tree is embedded into the plane in such a way that, for the root and for each  internal vertex of the shuffle tree, the inputs increase from left to right. 
We use the term~{\em canonical planar representative}~\cite[3.1]{dk} for such an embedding of a shuffle tree into the plane.
Following~\cite[8.2.3]{Loday_Valet}, we identify a  shuffle monomial or a shuffle tree
with its canonical planar representative.
Note that a canonical planar representative of a shuffle monomial is a planar tree whose leaves are enumerated by an initial segment of positive integers.
The enumeration of the leaves is defined by a substitution of a particular kind. So, we 
call  an enumeration of leaves of a planar tree monomial a  {\em shuffle substitution} if,
for the root and for each  internal vertex of the tree monomial, 
the inputs increase from left to right. Thus, the shuffle monomials are the planar tree monomials whose leaves are enumerated by the shuffle substitutions.\footnote{We are sorry to note that the term {\em shuffle substitution} is also used in~\cite{dk, Loday_Valet} with other meanings.}


Similarly to the non-symmetric case, there are two languages to think about the tree monomial.
We still recommend the reader to think about shuffle tree-monomials as labeled planar trees but most of the formulas
in examples are written in the more compact language of operations.
See Figure~\ref{pic::shuffle_monom} below for the comparison of these two languages, with the monomials
$g(f(f(x_1,x_3)$, $g(x_2,f(x_4,x_9)$, $g(x_5,x_6,x_{11}))),x_7,f(x_8,x_{10}))$, and
$ f(x_1,g(x_2,x_3,x_4))$ represented as trees.
\begin{wrapfigure}{r}{5.5in}
\caption{Divisibility of shuffle monomials}
\includegraphics{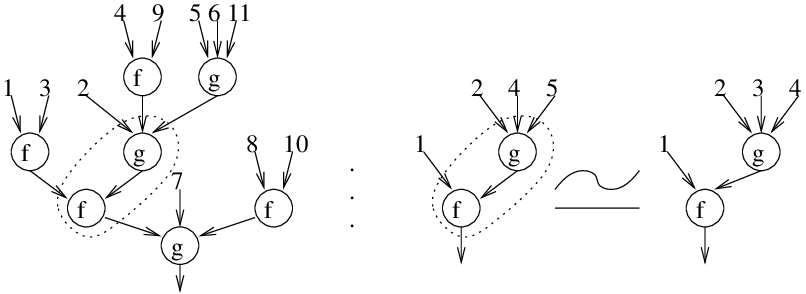}
\label{pic::shuffle_monom}
\end{wrapfigure}

Let us explain the notion of divisibility of shuffle trees.
Suppose that $u$ is a labeled subtree of  a shuffle tree $v$. 
Then one can consider $u$ as a shuffle tree if we assign to each input $i$ of $u$
the minimum $n_i$ of the leaves
of the subtree $v_i$ of $v$ grafted to $i$. 
We say that a monomial $v$ is divisible by a monomial $w$ if there exists a labeled subtree $u$ of $v$ which is isomorphic to $w$ as a shuffle tree, that is, 
$u$ is isomorphic to $w$ as a labeled tree and the isomorphism induces also the isomorphism of the ordered sets of leaves of $u$ and $w$. 

Figure~\ref{pic::shuffle_monom} provides an example of  divisibility of shuffle monomials.
In the  shuffle monomial represented by the left tree  we encircle a subtree with one internal edge.
This subtree is a divisor of the left monomial.
The tree in the middle is the encircled divisor where the leaves are numbered by the minima of the corresponding subtrees.
The right tree represents a shuffle monomial where we put the subsequent numbers on the leaves according to their local ordering.
Thus we see that the shuffle monomial
$$g(f(f(x_1,x_3),g(x_2,f(x_4,x_9),g(x_5,x_6,x_{11}))),x_7,f(x_8,x_{10}))$$
 is divisible by
$$ f(x_1,g(x_2,x_3,x_4)).$$

Let us define two particular cases of divisibility. We say that a
shuffle monomial $v$ is {\em right divisible} by a shuffle
monomial $w$ if $v$ is divisible by $w$ and, in the notation
above, all external vertices of the subtree $u$  are also external
as vertices of $v$. This means that there exists another shuffle
monomial  $t$ such that the monomial $v$ is obtained from a shuffle subtree 
$\widetilde t$ of $v$
isomorphic to $t$ by
replacing one of the  inputs  of $\widetilde t$ (say, the $m$-th) by a shuffle subtree
isomorphic to $u$. In this situation, we call the monomial $v$ a
{\em shuffle composition} of $t$ and $w$ and write $v = t
\circ_{m}^\sh w$. 
Contrary to the nonsymmetric case this notation is insufficient to
define the composition uniquely  because the former  should contain also 
information on ``shuffling'' substitution of the inputs of $t$ and
$w$~\cite[Prop.~2]{dk}.

Note that our shuffle compositions correspond to partial shuffle products 
defined in~\cite[8.2.5]{Loday_Valet}. Note also that our shuffle composition is a particular case of  shuffle composition defined in~\cite[8.2.6]{Loday_Valet} (with trivial all 
grafted trees but the $m$-th one).

Analogously, a
shuffle monomial $u$ (say, of arity $m$) is a {\em left} divisor of  $v$ if there is
a  labeled subtree $u$ inside $v$ with common root with $v$
which is isomorphic to $w$ as a labeled tree and with the local ordering
of the inputs with the same property as above.
In this case one can represent $v$ as a multiple shuffle composition
\begin{equation}
\label{eq:shuf_comp_intro}
v = (\dots (u
\circ_{m}^\sh v_m)
\circ_{m-1}^\sh v_{m-1}) \dots
\circ_{1}^\sh v_1)
\end{equation}
for some shuffle monomials $v_1, \dots, v_m$.
We denote this multiple composition simply by
$
u(v_1,\ldots,v_m)_{Sh}$.
Again, we will provide additional information on the permutation of inputs of $v$
whenever it is required.

\subsection{Monomial operads and Gr\"obner bases}

Let $\Upsilon$ be the set of {\em generators} of a non-symmetric or
shuffle operad $\p$, that is, each element of $\p$ is a linear
combination of compositions of elements of $\Upsilon$ with eachother.
Each space $\p(n)$ of $n$-ary operations has a
natural structure of a quotient of the component
$\f(\Upsilon)(n)$ of the (non-symmetric or symmetric) free operad $\f(\Upsilon)$
by some vector space $C(n)$. Then the suboperad $C=\cup_{n\ge
1}C(n)$ in $\f(\Upsilon)$, being a kernel of the natural surjection $\f(\Upsilon) \to
\p $, forms an ideal in $\f(\Upsilon)$ called the {\em ideal of
relations} of $\p$.
 The operad $\p$ is called {\em monomial} if each vector space $C(n)$ is spanned by monomials
(non-symmetric or shuffle, according to the type of the operad
$\p$). In this case, the ideal $C$
 is called monomial as well.
For a monomial operad $\p$, all monomials which do {\em not}
belong to $C(n)$ form a linear basis of $\p(n)$ for each $n$.

Given an order on monomials compatible with compositions
(for the discussion on admissible orderings, see~\cite[\S3.2]{dk}),
one can define a {\em leading monomial} for each element of the free operad.
For  the ideal $C$ of the relations of an operad $\p$ as above,
let $\widehat C$ be the span of the  leading monomials of all its elements.
Then $\widehat C$ is a monomial ideal such that the monomial basis of the quotient
operad $\widehat \p = \f(\Upsilon)/\widehat C $ is also a basis of $\p$ (if one identifies
the monomials in $\f(\Upsilon)$ with their images in $\p$).
For instance,  the dimensions of the corresponding components of $\p$ and $\widehat \p$
are the same, so that these two operads have the same generating series.
A {\em Gr\"obner basis}
of the ideal $C$ is a set of  elements of $C$ such that their leading monomials
generate the monomial ideal $\widehat C$~\cite[\S3.5]{dk}.
Gr\"obner basis gives a way to construct the monomial operad $\widehat \p$ starting from $\p$.
The operad $\widehat \p$ is called the \emph{monomial replacement} of $\p$ and has the same
generating series of dimensions as $\p$.

Consequently, if an operad $\p$ admits a finite  Gr\"obner basis
then the generating series of $\p$ is the same as the generating series of
the operad $\widehat \p$ defined by a finite set of monomial relations
(i.e., the leading terms of the elements of the Gr\"obner basis).
In this paper, we study the generating series of such operads $\p$.
The operads which admit a finite Gr\"obner basis include
PBW operads (\cite{hof},~\cite[Cor.3]{dk}),
operads coming from commutative algebras~\cite[\S4.2]{dk}
and many others. Some new nonquadratic examples are presented below.


\section{Nonsymmetric operads}

\label{sec::nonsym}

\subsection{Generating series and compositions}

\label{sec::nonsym_grading}

Suppose that the subset $\M\subset \B(\Upsilon)$ defines a monomial basis of a non-symmetric operad
$\P:=\F(\Upsilon)/(\Phi)$ (where $(\Phi)$ denotes the ideal  in the free non-symmetric
operad $\F(\Upsilon)$ generated by the set $\Phi\subset \F(\Upsilon)$ called the set of relations of $\P$).
The (ordinary) generating series of $\P$
is defined as the generating function of the dimensions of its components:
$$
G_{\P} (z) := \sum_{n\geq 1} \dim \P(n) z^{n} = \sum_{v\in \M} z^{ar(v)}.
$$

Moreover, if there exists an additional grading of the set of generators $\Upsilon$
such that all relations from $\Phi$ are homogeneous with respect to this internal grading, one can
consider in addition  a generating series with two parameters
$$
G_{\P} (z,t) := \sum_{n\geq 1} \dim_{t} \P(n) z^{n} = \sum_{v\in \M} t^{|v|}z^{ar(v)}
$$
where $\dim_t$ is the graded dimension.
We will omit the parameter $t$ if the internal grading is defined by the arity
of operations.
For example, if all generating operations have degree $1$ with respect to the internal grading and
have same arity $k>1$, then all operations of arity $n(k-1) +1$ should have the same grading $n$.
However, if there are unary generators or there are different generators of different arities 
the graded dimension $\dim_t$ becomes very important.

Let us consider in detail the generating series of one composition.
Namely, let $\mu$ be an $m$-ary  generator in the free operad $\F(\Upsilon)$
(i.e., $\mu \in {\Upsilon}_m$),  and let $\P_1, \dots, \P_m$  be some graded vector subspaces in $\F(\Upsilon)$.
\emph{A non-symmetric composition $\mu(\P_1,\ldots,\P_m)$ of vector spaces}  $\P_1, \dots, \P_m$
is the subspace of $\F(\Upsilon)$
whose $n$-th component is spanned by all possible compositions
$$
\mu(p_1,\ldots,p_m),\text{ where }p_i\in \P_i(k_i) \text{ and } \sum k_i = n.
$$

\begin{lemma}
 The generating series of a non-symmetric composition  $\mu(\P_1,\ldots,\P_m)$
is the product of generating series of vector spaces $\P_1,\ldots,\P_m$:
$$
G_{\mu(\P_1,\ldots,\P_m)}(z,t) = t^{|\mu|}\left( G_{\P_1}(z,t)\cdot\ldots\cdot G_{\P_m}(z,t)\right).
$$
\end{lemma}
\begin{proof}
Each monomial in the free non-symmetric operad is a concatenation of the root vertex and the collection of 
monomials which correspond to the subtrees attached to the inputs of the root vertex.
Moreover, this presentation is unique, therefore there is a canonical isomorphism of graded vector spaces:
$\mu(\P_1,\ldots,\P_m) \simeq \P_1\otimes \ldots \otimes \P_m.$
\end{proof}

\subsection{System of equations for generating series}

\label{sec:nonsym_2ndproof}

We give here a result which is a key point of Theorem~\ref{th-nonsym-intro}.
It gives a system of algebraic equations which will lead (via elimination of variables) to the algebraic equation
for $G_{\P}$, see Subsection~\ref{sec:alg_eq_nonsym}.
The proof of this result given here contains some core algorithms
enabling computations in particular examples.

\begin{theorem}
\label{th:nonsym_mon_difurs}
 For a given non-symmetric operad $\P$ with a finite set of generators and a finite Gr\"obner basis there exist an integer $N$
 and a system of algebraic equations on $N+1$ functions $y_0 = y_0(z,t),$\ldots, $y_N = y_N(z,t)$
\begin{equation}
\label{eq::nonsym_system_eq}
y_i =  t^{a_i}\sum_{s\in [0..N]^{d_i}} q^i_s\cdot
     y_{s_1}\cdot \ldots\cdot  y_{s_{d_i}}            \phantom{qqq} \mbox{ for }i =1, \dots, N  ,
\end{equation}
such that $G_{\P}(z,t) = \sum_{i=0}^{N}y_i(z,t)$, $y_0 = z$ and $y_i(0,t)=y_i(z,0)=0$, $i=1\ldots N$.
The numbers $q^i_s\in \{0,1\}$  and the nonnegative integers ${d_i},{a_i}$ and $ N $
are bounded from above by some functions of the degrees and
the numbers of generators and relations of the operad $\P$.
\end{theorem}

Note that, under the initial conditions the solution of
 system~(\ref{eq::nonsym_system_eq}) is unique.
 It follows that one can consider~(\ref{eq::nonsym_system_eq}) as a system of recursive
equations on the coefficients of the series $y_i$. We will refer to~(\ref{eq::nonsym_system_eq})
as the system of recursive equations.

Below we present two different
algorithms to construct the above system of equations (in Subsections~\ref{sec::many_relations} and~\ref{sec::few_relations}).
In addition, we give also an idea of a third algorithm.

The first one is more efficient in the case of many relations of relatively
low arity. Our proof of Theorem~\ref{th:nonsym_mon_difurs} is based on this algorithm.
Additionally, it gives the positivity of coefficients
$q^i_s$ which we use in Corollary~\ref{cor::nonsym_constant_growth}.
The second algorithm seems much more useful in the case of few
relations. Namely, the number $N$ of additional variables is
typically much lower than in the first algorithm. On the other
hand, the second algorithm does not allow to have only positive
coefficients $q^{i}_s$, that is, it gives a slightly weaker
version of  Theorem~\ref{th:nonsym_mon_difurs}.
Moreover, in Subsection \ref{sec::homology_algorithm} we
present an example of computations based on a third idea, that is,
we use combinatorics of homology of the operad.
We hope that a reader can get the general idea of this method
from the given example in order to compute the functional inverse to the
generating series of  a given operad.

Note that
the systems of equations similar to
the system~\eqref{eq::nonsym_system_eq} above sometimes appear in combinatorics and computer science
in
symbolic methods for combinatorial structures (see e.g.~\cite{anal_comb})
 as well as for the context-free languages (see e.g.~\cite{Ch_Sh}).

\subsubsection{The proof of  Theorem~\ref{th:nonsym_mon_difurs}}
\label{sec::many_relations}

\begin{proof}
Suppose that an operad $\P$ has a finite set of generators $\Upsilon$
and a finite set of monomial relations $\Phi$.
(It suffices to consider monomial relations since we are dealing with generating series.
Therefore, there is no difference between  the relations that form a Gr\"obner basis
and corresponding monomial relations presented by the leading terms of the Gr\"obner basis.)
 Let $d$ be
the maximum level of leaves of elements of $\Phi$
(by the \emph{level} of a
vertex/leaf in a tree we mean the number of the edges in the path from the
root to this vertex/leaf).
As mentioned in Section~\ref{lem::basis_in_nonsym}
every monomial $v$ in the free operad $\F(\Upsilon)$ can be identified
with a rooted planar tree whose vertices are
marked by elements of $\Upsilon$.
Given such a
monomial $v$, by its {\it stump} $b(v)$ we mean its maximal
monomial left divisor such that the leaves of $b(v)$ have levels
strictly less than $d$.
In other words, $b(v)$ is the
submonomial (rooted subtree) of $v$ which consists of the root and
all vertices and leaves of $v$ of level less than or equal to $(d-1)$
 and all edges connected to them.

Let $\Stump$ be the set of all stumps of all nonzero monomials in $\P$.
Let $N$ be the cardinality of this set.
The elements $b_1, \dots, b_N$ of $\Stump$ are partially ordered by the
following  relation:
$$b_i < b_j\text{ iff }i\ne j\text{ and }b_i
\text{ is a left divisor of } b_j.$$
 Let $\M_{b_i}$ be the set of all monomials in (= the
monomial basis of) the right-sided ideal $(b_i)\subset \P$
generated by $b_i$, and let
$$
\widetilde{\M_{b_i}} = \M_{b_i} \setminus \bigcup_{j: b_i<b_j} \M_{b_j}.
$$
Then the pairwise intersections of the sets
$\widetilde{\M_{b_i}}$  are empty.
Moreover, the disjoint union $\bigcup_{i=1}^N \widetilde{\M_{b_i}} $ is the
monomial basis of the operad $\P$. We have
$$
G_{\P}(z) = \sum_{i=1}^N y_i(z),
$$
where $y_i(z) = G_{span(\widetilde{\M_{b_i}})}(z)$ is the generating series of the span
 of the set $\widetilde{\M_{b_i}}$. For every element (=operation)
$\mu \in \Upsilon$ of some arity $n$, let us define the numbers
$j_\mu(i_1, \dots, i_n)$ for all $1\le i_1, \dots, i_n\le N$ as
follows:
\begin{equation}
 \label{eq::stump_requr_index}
       j_\mu(i_1, \dots, i_n) = \left\{
        \begin{array}{ll}
                   0, & \mbox{ if } \mu(\widetilde{\M_{b_{i_1}}}, \dots, \widetilde{\M_{b_{i_n}}}) =0 \mbox{ in } \P
                   \\
                   j, & \mbox{ if the stump } b( \mu(\widetilde{\M_{b_{i_1}}}, \dots, \widetilde{\M_{b_{i_n}}})) =
                   b_j.
        \end{array}
                  \right.
\end{equation}

Note that the nonzero sets of the type  $\mu(\widetilde{\M_{b_{i_1}}}, \dots,
\widetilde{\M_{b_{i_n}}})$ have empty pairwise intersections.
Let $v$ be a nonzero monomial in $\P$ with the root vertex labeled by $\mu$.
 Then $v \in \mu (\widetilde{\M_{b_{i_1}}}, \dots,
\widetilde{\M_{b_{i_n}}})$ for some $\mu, i_1, \dots, i_n$, that is,
$v = \mu(v_{i_1},\ldots,v_{i_n})$
where the monomial subtrees  $v_{i_j} \in \widetilde{\M_{b_{i_j}}}$ are uniquely determined by
$v$.
Hence, $v\in \widetilde{\M_{b_j}}$ where $j = j_\mu(i_1, \dots, i_n)$ from~(\ref{eq::stump_requr_index}).
As soon as the degrees of the relations are less than or equal to $d$ we come up with
the following disjoint union decomposition
for all $j=1\ldots N$
$$
           \widetilde{\M_{b_{j}}} =  \bigcup_{j_\mu(i_1, \dots, i_n) = j} \mu(\widetilde{\M_{b_{i_1}}}, \dots,
                \widetilde{\M_{b_{i_n}}}), \text{ where } \mu \text{ is the root vertex of any } v\in \widetilde{\M_{b_{j}}}.
$$
This equality implies equations~(\ref{eq::nonsym_system_eq}) for the generating functions
$y_i(z) = G_{span(\widetilde{\M_{b_i}})}(z)$.
\end{proof}

Notice that if the stump $b_j$ does not have leaves of level $(d-1)$ then the corresponding set $\widetilde{\M_{b_j}}$
consists of one element $b_j$. For example,
$\M_{1}$ always consists of one element representing the identity operation
and $G_{\M_{1}} = z$. In fact, the substitution  $y_0=G_{\M_{1}} = z$
is already made in the system~(\ref{eq::nonsym_system_eq}).
Therefore, one can reduce the number of algebraic equations in the system~(\ref{eq::nonsym_system_eq})
to the number of the stumps of level $(d-1)$
which contain at least one leaf of level $(d-1)$.
In particular, if the operad $\P$ is a PBW operad then  the number of recursive equations
is one greater than the number of generators.


\begin{example}
\label{exam:assoc}
 Let $\ass$ be the operad of associative algebras considered as a non-symmetric operad.
Namely, $\ass$ is generated by one binary operation $\mu(,)$
subject to  one quadratic relation $\mu(\mu(a,b),c) = \mu(a,\mu(b,c))$.
It is well known that this relation forms a Gr\"obner basis
according to the standard lexicographical ordering of monomials~\cite{dk,hof}.
The set $\Stump$ of stumps  of level less than or equal to $1$
consists of the identity operator $1$ and the binary operation $\mu$.
So, there are two sets of the type $\widetilde{\M}$, that is, $\widetilde{\M_{1}} = \{ 1\}$ and $\widetilde{\M_{\mu}}$.
Thus, we get the following system of equations:
$$
\begin{cases}
 G_{\ass} = z+y_\mu, \\
 y_\mu = z^2 + z y_\mu
\end{cases}
\Longrightarrow G_{\ass} = z + z G_{\ass} \Longrightarrow G_{\ass}(z) = \frac{z}{1-z}.
$$
\end{example}

\subsubsection{Decreasing the number of equations}
\label{sec::few_relations}

Now we present one more algorithm that allows to derive a system
of equations similar to~(\ref{eq::nonsym_system_eq}).
In the system produced by this algorithm  the number of equations in many examples 
is much less than in the algorithm  given in Subsection~\ref{sec::many_relations} above.
The cost of this is the following: in contrast to Theorem~\ref{th:nonsym_mon_difurs},
the coefficients $q^i_s$ might be negative
(i.~e., $q^i_s \in\{-1,0,1\}$).
One may consider this algorithm as a generalization of the algorithm for binary trees with one relation
presented in~\cite{Rowland}. 
While preparing the text we were informed by Lara Prudel 
that appropriate generalization of Rowlands algorithm for ternary trees has been presented in~\cite{3-ary_trees}.

The key point of the algorithm is the following.
 In the notation of
Subsection~\ref{sec::many_relations}, this new algorithm gives a system of equations in the
generating series of the sets $\M_{b_i}$ instead
of $\widetilde{\M_{b_i}}$. Then $G_P$ is equal to the sum of $z$
and  those $G_{\M_{b_i}}$ where $b_i$ runs over the generators of the operad $\P$.
 Therefore, to find $G_P$ it is
sometimes sufficient to solve only a part of the system which allows
 to express these new variables. Thus, the number of
equations (=the number of variables) of the reduced system
can be less than the number $N$ of equations (and variables)
of the system constructed in Subsection~\ref{sec::many_relations}.

\begin{proof}[An algorithm for constructing a system of equations]

Let $\P$ be as above, i.e. $\P$ is a finitely presented
(non-symmetric) monomial operad with a finite set of generators
$\Upsilon$ and a finite set of relations $\Phi$. We suppose that the
set of relations is reduced, namely, $g$ is not divisible by $g'$
for every distinct pair $g,g'\in \Phi$. Let $\M$ be the set of
monomials in the free operad $\F(\Upsilon)$ which are not divisible by the
relations from $\Phi$. In  other words, $\M$ is a monomial basis of
$\P$. We also suppose that the unary  identity operation $1$ belongs to
the set of monomials $\M$. We call it the {\it trivial} monomial.

Consider a free operad $\F(\Upsilon)$ and the set of all tree monomials
$\B(\Upsilon)$ generated by the same set $\Upsilon$.
To each monomial $v\in\B(\Upsilon)$ one can associate the subset $\F_v\subset \B(\Upsilon)$
consisting of those basis
elements that are left divisible by $v$.
In other words, the set $\F_{v}$ is a monomial basis of a right ideal
$v\circ \F$ generated by $v$.
Obviously, for the identity operation $1$, we  have $\F_1=\B(\Upsilon)$.
Each collection $\{v_1,\ldots,v_l\}$ of monomials defines \emph{the left common multiple}
 denoted by $[v_1\cup\ldots\cup v_l]$.
The monomial $[v_1\cup\ldots\cup v_l]$ is defined as the smallest element in
$\F_{v_1}\cap\ldots \cap \F_{v_l}$.
The left common multiple (if exists) should be the unique tree given as union of its subtrees $v_i$.
 As the definition of the left common multiple one may use the following identity:
$$
\F_{v_1}\cap\ldots \cap \F_{v_l} = \F_{[v_1 \cup\ldots\cup v_l]}
$$
For the case of  empty
 intersection  $\F_{v_1}\cap\ldots \cap \F_{v_l}$ we set $[v_1\cup\ldots\cup v_l]$ to be zero.

To each monomial $v\in\M$ we associate the monomial basis $\M_v$ of the corresponding right ideal $v\circ\M$.
Obviously, we have $\M_v = \M\cap \F_v$.

Consider a given nontrivial monomial $v\in\M$. Suppose that the root generator of $v$ is a $k$-ary operation $\mu$,
that is $v = \mu\circ(v^1,\ldots,v^k)$ where $v^i$ denotes the subtree which grows from the $i$-th incoming
arrow of the root vertex  of $v$.
Let $\Phi_{v}$ be the subset of the set of generating relations $\Phi$
such that the corresponding relations have a nontrivial left common multiple with $v$ in the free operad $\F(\Upsilon)$, i.e.
$$
g\in \Phi_v\subset \Phi \stackrel{def}{\Longleftrightarrow} [v\cup g]\neq 0\ \&\ g\in\Phi
 \Longleftrightarrow \F_{v}\cap\F_g \neq 0\ \&\ g\in\Phi
$$
If the left common multiple $[v\cup g]$ is non-zero then $\mu$
should be a left divisor of $v$, hence,
we have a decomposition $g = \mu\circ(g^1,\ldots,g^k)$ for some $g^1,\ldots,g^k$.
We get the following recursive relation for the set $\M_v$:
\begin{equation}
\label{monom::recurent}
 \M_v = \mu\circ (\M_{v^1},\ldots,\M_{v^k}) \setminus
\left[\bigcup_{g\in \Phi_v}\mu\circ\left( \M_{[v^1\cup g^1]},\ldots,\M_{[v^k\cup g^k]}\right)\right].
\end{equation}
The Identity~(\ref{monom::recurent}) should be clear from the observation that  each monomial from
$\mu\circ (\M_{v^1},\ldots,\M_{v^k})$
either belongs to $\M_v$ or is divisible from the left by  the relation $g\in \Phi_v$.

Denote by $y_v(z)$ (or $y_v(z,t)$ if the operad is $\Z^2$-graded)
the generating series of the span of the set $\M_v$.
Combining the formula~(\ref{monom::recurent}) with the inclusion-exclusion principle, we
 count the cardinalities and the generating series  of the sets $\M_v$ and their intersections. We have
\begin{equation}
y_v = t^{|\mu|} \left[ \sum_{s=0}^{|\Phi_v|} (-1)^s \left(
\sum_{\begin{smallmatrix}
       \{g_{1},\ldots,g_{s}\} \subset \Phi_{v} :\\
      [v\cup g_{1}\cup\ldots\cup g_{s}] \neq 0
      \end{smallmatrix}}
 \left( y_{[v^1\cup g_{1}^1\cup\ldots\cup g_{s}^1]}\cdot\ldots \cdot
y_{[v^k\cup g_{1}^k\cup\ldots\cup g_{s}^k]}
\right)
\right) \right],
\label{series::recurent}
\end{equation}
where (as above) $\mu$ is the generator placed in the root vertex of both $v$ and $g_j$,
and the subtrees growing from the $i$-th incoming arrow of $\mu$
are denoted by $v^i$ and $g_{j}^i$ respectively.
 The internal summation is taken
over the subsets $\{g_{1},\ldots,g_{s}\} \subset
\Phi_v$ of cardinality $s$ such that the corresponding left common
multiple $[v\cup g_{1}\cup\ldots\cup g_{s}]$ is
different from zero and is not divisible by any element of
$\Phi$. 

So, we have constructed a system of the recursive equations~(\ref{series::recurent}).
To show that this algorithm is correct, it remains to bound the number of equations. 
 This is achieved in the next Lemma.

\begin{lemma}
 \label{lem::buts_few_relations}
 There exists a minimal finite set of monomials $T(\P)\subset \M$
satisfying the following conditions.
\begin{itemize}
 \item The identity operator $1$ and all generators $\Upsilon$
belongs to $T(\P)$.
\item
Suppose that $v=\mu(v^1,\ldots,v^k)\in T(\P)$ and a
collection $\{g_{1},\ldots,g_{s}\}\subset \Phi_v$
has a nontrivial left common multiple $w = [v\cup
g_{1}\cup\ldots\cup g_{s}]\neq 0$ with $v$ (in particular, each $g_{j}$
has the form $g_j = \mu(g_j^1,\ldots,g_j^k)$). Then
for each $i=1,\ldots,k$ the monomial $[v^i\cup
g_{1}^i\cup\ldots\cup g_{s}^i]$
either
belongs
to $T(\P)$  or is divisible by a relation from $\Phi$.
\end{itemize}
\end{lemma}
\begin{proof}
Suppose that $d>1$ is the maximum of the levels of the relations in $\Phi$.
Since the level of the left common multiple $[v\cup w]$ is bounded from
above by the maximum of the levels of $v$ and $w$,
we conclude that the level of any monomial from $T(\P)$ is strictly less than $d$.
The set of monomials with bounded level is finite thus $T(\P)$ is finite.
\end{proof}

In fact, the bound of $T(\P)$ by the number of monomials in $\P$ of level less than or equal to $(d-1)$ is quite large. 
We will see in examples below that the cardinality of $T(\P)$
is much less.

It is clear that
\begin{equation}
y_1 = z + \sum_{\mu\in \Upsilon} y_{\mu}
\label{series::1}
\end{equation}
Therefore, for $v\in T(\P)$ we get a system of algebraic equations~(\ref{series::1}) and~(\ref{series::recurent})
using a finite set of unknown functions $\{y_v |v\in T(\P)\}$.
\end{proof}

\begin{example}
\label{ex::Asw}
 Let $\Q$ be a non-symmetric operad generated by one binary operation $(,)$
 satisfying the following two  monomial relations
of arities $4$ and $5$:
\begin{equation} \label{eq::weak-ass}
((a, b), c), d) =0 \text{ and }  (a,(b,((c,d),e))) = 0.
\end{equation}
The set $T(\Q)$ from Lemma~\ref{lem::buts_few_relations} consists of the following $5$ elements:
$$
T(\Q):= \{ 1 ; (ab) ; ((ab)c) ; (a((bc)d)) ; ((ab)((cd)e)) \}.
$$
The corresponding system of recursive equations is
\begin{equation}
\label{eq::Asw::rec:rel}
\left\{
\begin{array}{l}
y_1 = z+ y_{(ab)} , \\
y_{(ab)} = y_1^2 - y_{(ab)} y_1 - y_1 y_{((ab)c)} + y_{((ab)c)} y_{(a((bc)d))}, \\
y_{((ab)c)} = y_1 y_{(ab)} - y_{((ab)c)} y_1 - y_{(ab)} y_{(a((bc)d))} + y_{((ab)c)} y_{(a((bc)d))}, \\
y_{(a((bc)d))} = y_1 y_{((ab)c)} - y_{((ab)c)}^2 - y_1 y_{((ab)((cd)e))} + y_{((ab)c)} y_{((ab)((cd)e))}, \\
y_{((ab)((cd)e))} =
 y_{(ab)} y_{((ab)c)} - y_{((ab)c)}^2 - y_{(ab)} y_{((ab)((cd)e))} + y_{((ab)c)} y_{((ab)((cd)e))}.
\end{array}
\right.
\end{equation}
To exclude additional variables we make the following linear change of variables:
$y:=y_1$, $v_3:=y_1-y_{((ab)c)}$, $v_4:=y_1-y_{(a((bc)d))}$, $v_5:=y_{((ab)c)}-y_{((ab)((cd)e))}$.
The system~(\ref{eq::Asw::rec:rel})
is then equivalent to the following
\begin{equation}
\left\{
\begin{array}{l}
 y-z = v_3 v_4, \\
y- v_3 = (v_3 - z) v_4 \Rightarrow v_3 -z = z v_4 ,\\
y - v_4 = v_3 v_5 \\
y - v_3 -v_5 = (v_3 - z) v_5 \Rightarrow v_3-v_4 +v_5 = z v_5 .
\end{array}
\right.
\begin{array}{c}
\Rightarrow
\left\{
\begin{array}{l}
 y-z = z(1+v_4) v_4, \\
y - v_4 = z(1+v_4)(v_4 - \frac{z}{1-z})
\end{array}
\right. \\
\Rightarrow G_{\Q}(z)= y = z+ \frac{z^2(1-z^2)}{(1-z -z^2)^2} .
\end{array}
\end{equation}

\end{example}

\subsubsection{Computations via homology}
\label{sec::homology_algorithm}

In this section we show how one can simplify computations in some cases
using the monomial resolutions of operads with finite Gr\"obner bases introduced in~\cite{dk_resolutions}.
First, we recall a description of a monomial basis in these resolutions.
Second, we present a particular example where such a description
allows to compute the generating series.
The corresponding computation using the first two methods (given in two previous examples)
became extremely hard compared to what the homological method can give.
Theoretically, the monomial description of a resolution from \cite{dk_resolutions}
allows to produce an algorithm similar to the one given in Section~\ref{sec::many_relations}
 starting from a given finite Gr\"obner basis.
But the combinatorics involved  becomes tricky as soon as the complexity of intersections of
leading terms of monomials grows up.
Therefore, we decided not to give all the details of this algorithm.
However, a couple of examples given for the non-symmetric (Example~\ref{ex::nonsym_homol_exam})
 and the symmetric (Example~\ref{ex::sym_homol_exam}) cases should convince the readers that
in some computations the homological method may be more effective.

Suppose that $R$ is a free resolution (=DG model) of an operad $\P$ generated by a differential
graded vector space $\Q$.
(In other words, $R\simeq \F(\Q)$ as an operad  and there
exists a differential $d$ on $R$ such that the homology operad of $R$ is isomorphic to $\P$).
Then the generating series of $\P$ and $\Q$ are related by the equality
\begin{equation}
\label{eq::inverse}
G_{\P}(z) - G_{\Q}(G_{\P}(z)) = G_{\P}(z) - G_{\P}(G_{\Q}(z)) = z,
\end{equation}
where $G_{\Q}(z)$ is the generating  series of the Euler characteristics
$\chi$ of the components  $\Q$:
$$
G_{\Q}(z) =\sum_{n\geq 1} \chi(\Q(n)) z^{n}.
$$

Let $\P$ be a finitely presented operad with a set of generators $\Upsilon$
and a set of monomial relations $\Phi$.
Let us recall a basis in a free monomial resolution of the operad $\P$.
\begin{proposition*}[see \cite{dk_resolutions}]
 There exists a free resolution $(R,d)\stackrel{quasi}{\twoheadrightarrow} \P$
such that the set of free generators of $R$ consists of the union of the set $\Upsilon$
and the set $\H$ elements which are numbered by the following pairs:
a monomial $v\in \B(\Upsilon)$ and a set $\{w_1,\ldots,w_n\}$ of labeled subtrees (=submonomials) of $v$
satisfying the  following two conditions.
\begin{itemize}
 \item[$(h1)$] Each $w_i$  is isomorphic to one of the elements of $\Phi$ as a planar labeled tree.
 \item[$(h2)$] Each internal edge of the monomial $v$ should be covered by at least one of subtrees $w_i$.
In other words, there is no decomposition $v=v'\circ v''$ such that each $w_i$ is a subtree of $v'$ or a subtree of $v''$.
\end{itemize}
The homological degree of each generator $x\in \Upsilon$ is set to be zero and the homological degree of
the  generator $(v,\{w_1,\ldots,w_n\})\in \H$ is set to be  $n$.
\end{proposition*}

Note that the monomial $v$ in a pair $(v,\{w_1,\ldots,w_n\})\in \H$
is uniquely defined by the set of submonomials $\{w_1,\ldots,w_n\}$.
For such a pair we will use a notation $\overline{w}$.

Let us present an example where we use this description of a basis
in a monomial resolution to get a functional equation for the generating series.

\begin{exam}\label{ex::nonsym_homol_exam}
 Consider a non-symmetric operad $\Q_k$ generated by one binary operation $(,)$
which satisfies the following relation of  degree $k>2$ (i.~e., of
arity $k+1$):
$$
r_k:=(x_1,(\dots ,( x_{k-2},(x_{k-1}, x_k, x_{k+1})) \dots )=0,
$$
where $(a,b,c)$ denotes  the associator
$(ab)c-a(bc)$.

The operad $\Q_2$ coincides with the operad $\ass$ of associative algebras.  So, one might consider
the identity $r_k=0$ of  $\Q_k$ for $k>2$ as a weak version of associativity.

\begin{proposition}
 The unique generator $r_k$
of the ideal of relations
forms a Gr\"obner basis of relations in $\Q_k$.
\end{proposition}

\begin{proof}
 Let us show that all possible $s$-polynomials reduce to zero.
Indeed, there are exactly $(k-1)$ intersections of the leading terms
of the relation $(r_k)$:
\begin{gather*}
S_{k+l+2}   :=
(x_1(x_2\overbrace{(\ldots((}^{k-2}x_{k-1}x_k)
(x_{k+1} \overbrace{(\ldots(}^{l-2}x_{k+l-1}((x_{k+l}x_{k+l+1})x_{k+l+2}\overbrace{))\ldots)}^{k+l-1}
 \quad \text{ for } l=1,\ldots,k-2
\\
\text{ and } \\ {S'_{2k}}  := (x_1(x_2\overbrace{(\ldots((}^{k-2}
x_{k-1}(x_k \overbrace{(\ldots(}^{k-4}
x_{2k-4}((x_{2k-3}x_{2k-2})x_{2k-1}\overbrace{)\ldots)}^{k+1}
x_{2k}\overbrace{)\ldots)}^{k-2}.
\end{gather*}
The lower index on the left hand side corresponds to the number of leaves/inputs in a monomial.
The corresponding  $s$-polynomials are as follows:
$$
S_{k+l+2} \rightsquigarrow (x_1(\ldots((x_{k+l}x_{k+l+1})x_{k+l+2})\ldots)) -   {
(x_1(\ldots((x_{k-1}x_k)x_{k+1}(\ldots(x_{k+l+1}x_{k+l+2})\ldots))\ldots))}
$$
\begin{multline*}
S'_{2k} \rightsquigarrow (x_1(\ldots(x_{k-1}((x_{k}(\ldots
(x_{2k-4}((x_{2k-3}x_{2k-2})x_{2k-1}))\ldots))x_{2k}))\ldots)) \\
-   { (x_1(\ldots(x_{k-2}((x_{k-1}(\ldots
(x_{2k-3}(x_{2k-2}x_{2k-1}))\ldots))x_{2k}))\ldots)) }.
\end{multline*}
It is easy to see that each monomial of the form
$$
  (x_1(\dots (x_{k-2}( f(x_k, \dots, x_j))\dots)),
$$
where $j\ge k+1$ and $f$ is an arbitrary  iterated composition of
the operation $(,)$, is reducible via $r_k$ to the monomial
$$
 m_j = (x_1(\dots (x_{j-2} (x_{j-1}, x_j))\dots)).
$$
Thus, both monomials in each $s$-polynomial above are reduced
to the same monomial $m_j$ for suitable choice of $j$. This means that each
$s$-polynomial is reduced to $m_j-m_j=0$.
\end{proof}

As soon as the Gr\"obner basis is chosen it remains to compute the generating series of homology
for the corresponding monomial replacement.

The operad $\Q_k$ is generated by one binary operation. Therefore,
the tree-monomials under consideration are rooted planar binary trees
where all internal vertices are labeled by the same operation $(,)$
and we will omit this labeling with no loss.
In order to specify the tree-type of monomials we say that
the monomial operation $((x_1 x_2)x_3)$ corresponds to a planar rooted binary tree with two internal
vertices and one internal edge that goes to the left
and the monomial operation $(x_1 (x_2 x_3))$ corresponds to a binary planar rooted tree with
one internal edge that goes to the right.
In particular,
the leading term of the unique element $(r_k)$ of a Gr\"obner basis corresponds to a planar rooted binary tree of level $k$
with $k$ internal vertices and the unique path that contains all $(k-1)$ internal edges of a tree.
This path starts at the root vertex, then follows the edge going to the right at each of the next k-2 vertices, 
then follows the edge going to the left at the last vertex.

For any given element $(v,\{w_1,\ldots,w_n\})\in \H$
all $w_i$ are isomorphic to the leading term of the relation $r_k$ as a planar binary tree.
Therefore, there exists exactly one submonomial
which contains the root vertex of $v$. Without loss of generality we assume that this submonomial is $w_1$.
The pair $(v,\{w_2,\ldots,w_n\})$ will no longer satisfy the property $(h2)$ of the elements in $\H$
but can be presented as the composition of a generator $(,)$ taken several times and pairs from $\H$.
In other words, one has to present the decomposition of the set of submonomials $\{w_2,\ldots,w_n\}$ into a disjoint union
of subsets such that $w_i$ and $w_j$ belong to the same subset if and only if there exists a submonomial
$v'$ in $v$ and a subset $\{w_i,w_j,\ldots\}\subset  \{w_2,\ldots,w_n\}$ such that the pair
$(v',\{w_i,w_j,\ldots\})$ is isomorphic to a pair from $\H$.
Let us show that in the case of the operad $\Q_k$ this decomposition contains at most two subsets.
Indeed the unique left internal edge of a submonomial $w_1$ may not belong to any other submonomial $w_i$ for $i>1$.
Consider a decomposition of a monomial $v = v'\circ v''$ according to this left internal edge.
Then each $w_i$ is a submonomial of $v'$ or $v''$ and we have a decomposition
\begin{equation}\label{eq::decomp_homology_r_k}
\{w_2,\ldots,w_n\} =
\{w_{\sigma(2)},\ldots,w_{\sigma(l)}\}\sqcup \{w_{\sigma(l+1)},\ldots,w_{\sigma(n)}\}
\end{equation}
for an appropriate permutation $\sigma$. 
The cases in which the first or second subset is empty may also occur.
There exists a pair of a uniquely defined (probably empty) submonomials $v_1\subset v'$ and $v_2\subset v''$
such that both pairs $(v_1,\{w_{\sigma(2)},\ldots,w_{\sigma(l)}\})$ and $(v_2,\{w_{\sigma(l+1)},\ldots,w_{\sigma(n)}\})$
belongs to $\H$.
Moreover, the monomial tree $v$ can be uniquely presented as a composition of several generators and monomials $v_1$ and $v_2$.
We get a recursive formula for the generating series $G_{\H}(z)$ of the Euler characteristics of elements in $\H$:
$$
G_{\H} = -z^{k+1} - \sum_{l=2}^{k-1} z^{l+1} G_{\H} - \sum_{l=1}^{k-2} z^{l} G_{\H}^2  - z^{k-1} G_{\H}.
$$
Here the first summand corresponds to the empty set $\{w_2,\ldots,w_n\}$;
the second summand corresponds to the empty set $\{w_{\sigma(l+1)},\ldots,w_{\sigma(n)}\}$;
the index $l$ corresponds to the number of internal vertices in a subtree $v'$ which do not belong to none of the vertices of
submonomials in $\{w_{\sigma(2)},\ldots,w_{\sigma(l)}\}$;
the third summand deals with both nonempty sets in  decomposition~(\ref{eq::decomp_homology_r_k});
and the fourth summand corresponds to the empty set $\{w_{\sigma(2)},\ldots,w_{\sigma(l)}\}$.
The powers of $z$ are equal to the number of leaves coming from the internal vertices that do not belong to submonomials
$v_1$ and $v_2$. The minus signs comes from the homological degree since we remove exactly one element $w_1$
from the set of submonomials.

Finally we have the following quadratic equation for the functional inverse series $G_{\Q_k}(z)^{-1} = z - z^{2} - G_{\H}$:
\begin{multline*}
\left( G_{\Q_k}(z)^{-1}\right)^2 (z^{k-1}-z) + G_{\Q_k}(z)^{-1} (z^{k+1} - 3z^{k} + z^{k-1} - z^{3} + 2 z^{2} - z +1) \\
-(z^{k+2} - 2 z^{k+1} + z^{k} - z^{4} + 2z^3 - 2z^2 +z ) = 0 ,
\end{multline*}
which is equivalent to an algebraic equation of degree $(k+2)$ on the generating series $G_{\Q_k}$.

\end{exam}

\subsection{Single algebraic equation for  generating series}

\label{sec:alg_eq_nonsym}

The classical elimination theory implies the existence of an algebraic equation on a function $G_{\P}$
from the system~(\ref{eq::nonsym_system_eq}) (see explanation below).
See also \cite{Ch_Sh} and the appendix B.1 in~\cite{anal_comb} and references therein where the same theorem is proven for
context-free specifications and languages.

\begin{theorem}
\label{th-nonsym}
 The generating series $G_{\P}$  of a non-symmetric operad $\P$ with a finite Gr\"obner basis is an algebraic function.
\end{theorem}

Starting from the system~(\ref{eq::nonsym_system_eq}), Theorem~\ref{th-nonsym}
immediately
follows from the next Lemma~\ref{lem:alg_elimination}.

\begin{lemma}
\label{lem:alg_elimination}
Suppose that the formal power series $f_1, \dots, f_n\in \QQ[[t,z]]$ without constant terms in variables $t$ and $z$
 satisfy a system of algebraic equations of the form
$$
 f_i  t^{\mu_i} = g_i (f_1, \dots, f_n)
$$
for each $i=1, \dots, n$, where $g_i$ is a homogeneous polynomial in $n$
variables of degree $d_i\ge 2$ and $\mu_i$ are positive integers.
Then the power series $f_1$
satisfies a polynomial equation
$$
Q(f_1) =0,
$$
where $Q$ is a non-constant polynomial with coefficients in $\QQ[t,z]$ such that
$\deg Q \le \left( \prod_{i=1}^n d_i \right)^2$.
\end{lemma}

\begin{proof}[Proof of Lemma~\ref{lem:alg_elimination}]
The above system has the form
$$
F = G(F),
$$
where $F = (f_1, \dots, f_n)$ and $G = (t^{-\mu_1} g_1, \dots, t^{-\mu_n} g_n)$,
or
$$
    H(F) = 0
$$
with $H = \Id - G$.
Note that  the Jacobi matrix  $J = \partial H / \partial F$ is  non-degenerate,
because $\det J =  1 + O(F) \ne 0$.
Let $K$ be the field of rational fractions $\QQ(t,z)$
and let $L$ be one of its algebraically closed extension
which contains the ring of formal power  series in $t$ and $z$. Obviously,
 the variety $V \subset L^n$
 of the solutions of the above system is 0-dimensional.

 Therefore,  there exists a non-trivial polynomial $T(x)$ over $L$ such that $T(f_1)=0$.
By Bezout's theorem, one can take $T$ such that $\deg T \le \prod_{i=1}^n d_i  $.
By effective Hilbert Nullstellensatz (see \cite[Corollary~1.7]{kollar}),
for some  $j\le \prod_{i=1}^n d_i $ the polynomial $ T(x)^j$ lies in the polynomial ideal $I$ generated over $K$ by
$G_i$'s. It follows that some divisor $Q(x)$ of $ T(x)^j$ belongs to the reduced  Gr\"obner basis of $I$
(w.~r.~t. the ``lex'' order).
  Since this element $Q(x)$ can be constructed via Buchberger's algorithm,
its coefficients belong to $K$. In addition, we have $Q(f_1) = 0$ and $\deg Q \le j \deg T \le \left( \prod_{i=1}^n d_i  \right)^2$.
\end{proof}

\begin{rema}
Note that the existence of such a polynomial $Q$ follows also from Artin's Approximation Theorem~\cite[Theorem~1.7]{artin}.
\end{rema}

\begin{corollary}
 Let $\P$ be a PBW operad with $k$ binary generating operations.
Then the generating series of this operad is a solution of an algebraic equation of degree not greater than $4^k$.
\end{corollary}

\begin{proof}
 The algorithm described above implies the existence of a system of $k$ quadratic and one linear equation on
$k+1$ functions. Then we apply Lemma~\ref{lem:alg_elimination}.
\end{proof}

\subsection{Non-symmetric operads of subexponential growth}

\label{sec::nonsum_const_growth}

Below we present an application of the above theory for non-symmetric operads
with small growth. Recall that a sequence $\{ a_n \}_{n\ge 0}$
of nonnegative real numbers is said to have  subexponential growth
if its growth is strictly less than exponential, that is, for each
$d>1$ there exists $C>0$ such that $a_n < Cd^n$ for all $n>0$.

\begin{corollary}
\label{cor::nonsym_constant_growth}
 Let $\P$ be a non-symmetric operad with a finite Gr\"obner basis of relations.
Suppose that the growth of dimensions $\P(n)$ is subexponential.
Then the ordinary generating series $G_{\P}$
 is rational. In particular, the
sequence of dimensions $\P(n)$ has a polynomial growth
$[n^d]$ for some integer $d$.
\end{corollary}

The proof of Corollary~\ref{cor::nonsym_constant_growth}
is based on the general facts about the Taylor coefficients of algebraic functions and on the positivity of
coefficients in the system~(\ref{eq::nonsym_system_eq}).

 Consider the system of equations~(\ref{eq::nonsym_system_eq}) 
obtained while using the first algorithm presented in Section~\ref{sec::many_relations}.
Let us remind that the series $y_i = \sum_{n\geq 0} y_{i,n} z^n$ is a generating series of
the set of monomials in $\widetilde{\M_{b_i}}$
and the right hand side of any equation from this system has strictly positive coefficients.
Therefore, for all $i$ the series $y_i$ is an algebraic function with nonnegative integer coefficients
bounded by the dimensions of $\P(n)$.
First, we explain the technical result on algebraic functions with subexponential growth of coefficients and then explain 
a proof of Corollary~\ref{cor::nonsym_constant_growth}.
\begin{lemma}
\label{lem::alg_func_coef}
 Suppose that $f(z):=\sum_{n\geq 1} f_n z^{n}$ is an algebraic function, such that the sequence of coefficients
form a sequence of nonnegative real numbers with subexponential growth.
Then there exists a rational number $s$, integer $m$ and a pair of constants $C_{-},C_{+}$ such that
for $n_0$ sufficiently large
\begin{itemize}
 \item[(a)] the upper bound $f_n< C_{+} n^s$ is true for all $n>n_0$,
 \item[(b)] the lower bound $f_n> C_{-} n^{s}$ is true for at least one index $n$ in each 
consecutive collection $\{N,N+1,\ldots,N+m-1\}$
of $m$ integers, where $N>n_0$.
\end{itemize}
\end{lemma}

Notice that the exponent $s$ in Lemma~\ref{lem::alg_func_coef} is one less then the so-called
Gelfand--Kirillov dimension
$$
\GKdim[f]:= \varlimsup_{n\to \infty} \frac{\ln\left(\sum_{i=0}^n f_i\right)}{\ln n}
$$
(see \cite{GK-dim} for more details about GK-dimensions of algebras).
Lemma~\ref{lem::exponents_units} below implies that the integer $m$ is bounded from above by the number of
singular points on the unit circle of the function $f$.

\begin{proof}
Let us recall the following well known theorem about asymptotic of coefficients of algebraic function
based on the Puiseux expansion near critical points.

\begin{theorem*}[see~{\cite[Theorem~D]{flj_slg_fn}}]
 If $f(z)=\sum_{n\geq 0} f_n z^{n}$ is an algebraic function over $\mathbb{Q}$ that is analytic at the origin, 
then there is the following asymptotics for the sequence of its Taylor coefficients:
$$
f_n= \beta^{n} n^s \sum_{i=0}^{m} C_i \omega_i^n +O(\beta^n n^{t}),
$$
where $s$ is a rational number, $t<s$, $\beta$ is a positive algebraic number
and $\omega_i$ are algebraic with $|\omega_i|=1$.
\end{theorem*}
It follows from the proof in \cite[Theorem~D]{flj_slg_fn}
that $\beta$ is the inverse of the radius of convergence.
Therefore in our case (where the sequence of nonnegative integer coefficients $f_n$ has subexponential growth)
 $\beta$ is equal to $1$.
The numbers $\omega_i$ are equal to the singular points on the unit circle of function $f(z)$.
The upper bound $(a)$ of Lemma~\ref{lem::alg_func_coef} follows from the upper bound:
$$
|\sum_{i=1}^{m} C_i \omega_i^n| \leq \sum_{i=1}^m |C_i| \
\Rightarrow \ f_n \leq (\sum|C_i|) n^s + O(n^t) < (1+\sum|C_i|) n^s \text{ for } n>>0.
$$
The following simple sub-lemma implies the lower bound $(b)$ in Lemma~\ref{lem::alg_func_coef}
and finishes its proof.
\end{proof}
\begin{lemma}
\label{lem::exponents_units}
For any given collection of distinct complex numbers $\omega_1,\ldots,\omega_m$ with $|\omega_i|=1$
and a given collection of constants $C_1,\ldots,C_m$
there exists a constant $C$ such that for all $n$ there exists an index $k=k(n)\in \{n,n+1,\ldots,n+m-1\}$
such that $|\sum_{i=1}^m C_i \omega_i^k|>C$.
\end{lemma}
\begin{proof}
The proof is by induction on $m$ (the number of summands).
The induction base is trivial since the absolute value $|C_1 \omega_1^n| = |C_1||\omega_1|^n = |C_1|$  does not depend on $n$. \\
\emph{Induction step.}
Consider an integer $n$. The induction hypothesis implies the existence of a
constant $C$ which does not depend on $n$ and an integer $k=k(n)\in \{n,n+1,\ldots,n+m-1\}$ such that
\begin{equation}
\label{eq::mod_root_unity}
\left|\sum_{i=1}^{m} \left(C_i \left(\frac{\omega_i}{\omega_{m+1}} -1\right)\right)
\left(\frac{\omega_i}{\omega_{m+1}}\right)^{k} \right| >2C .
\end{equation}
Therefore the absolute value of either the number
$\left(C_{m+1} +\sum_{i=1}^{m} {C_i} \left(\frac{\omega_i}{\omega_{m+1}}\right)^{k+1}\right)$
 or the next element of the sequence
$\left(C_{m+1} +\sum_{i=1}^{m} {C_i} \left(\frac{\omega_i}{\omega_{m+1}}\right)^{k+2}\right)$
is greater than ${C}$ since
their difference coincides with the left hand side of the inequality~(\ref{eq::mod_root_unity}).
The obvious equality of the absolute values
$$\left|\sum_{i=1}^{m+1} C_i \omega_i^n \right| =
 \left| C_{m+1} +\sum_{i=1}^{m} {C_i} \left(\frac{\omega_i}{\omega_{m+1}}\right)^{n}\right|$$
finishes the proof of the induction step.
\end{proof}
\begin{lemma}
\label{lem::sum_GK_dim}
 Let $f(z)$ and $g(z)$ be a pair of algebraic functions whose Taylor expansions at the origin have nonnegative
coefficients with subexponential growth.
Then the Gelfand--Kirillov dimension of the product $f(z)g(z)$ is the sum of the
Gelfand--Kirillov dimensions of the two factors:
$$
\GKdim[f(z) g(z) ] = \GKdim[f(z)]+\GKdim[g(z)]
$$
\end{lemma}
\begin{proof}
Let us denote the $n$-th Taylor coefficient of the product $f(z)g(z)$ by $(fg)_n$.
We have  $(fg)_n = \sum_{i+j=n} f_i g_j$.
 The upper bound on the Gelfand--Kirillov dimension is obvious:
\begin{multline*}
\sum_{i=1}^{n} (fg)_i < \sum_{i+j\leq n}C_{+}(f) i^{\GKdim[f]-1} C_{+}(g)j^{\GKdim[g]-1} \leq \\
\leq \frac{n(n+1)}{2} C_{+}(f) C_{+}(g) n^{\GKdim[f]+\GKdim[g]-2} < C_{+} n^{\GKdim[f]+\GKdim[g]},
\end{multline*}
where $C_+(f),C_{+}(g)$ are the constants from the upper bound in Lemma~\ref{lem::alg_func_coef}
for the functions $f$ and $g$ respectively.
We use a similar sequence of inequalities valid for all sufficiently large $n$ in order to prove the lower bound:
\begin{multline*}
\sum_{i=1}^{n} (fg)_i = \sum_{i+j\geq n} f_i g_j
\geq \sum_{\frac{n}{4} \leq i,j \leq \frac{n}{2}} f_i g_j \geq \\
\geq
\lfloor \frac{n}{4}\rfloor^2
C_{-}(f) \left(\lfloor \frac{n}{4 m_f} \rfloor\right)^{\GKdim[f]-1}
 C_{-}(g) \left( \lfloor \frac{n}{4 m_g} \rfloor \right)^{\GKdim[g]-1}
\geq C_{-} n^{\GKdim[f]+ \GKdim[g]},
\end{multline*}
where $(C_-(f),m_f)$ and $(C_-(g),m_g)$ are the constants and integers from
the lower bound in Lemma~\ref{lem::alg_func_coef} chosen for the algebraic functions $f$ and $g$ respectively.
\end{proof}

Finally, we can prove Corollary~\ref{cor::nonsym_constant_growth}.

\begin{proof}
Let $\P$ be a non-symmetric operad with a finite Gr\"obner basis
such that the sequence of dimensions $\dim\P(n)$ has subexponential growth.

Consider the system of equations~(\ref{eq::nonsym_system_eq}).
We know that the functions $y_i(z)$ are algebraic with subexponential growth of coefficients.
Lemma~\ref{lem::alg_func_coef} implies the existence of finite nonzero GK-dimension of all infinite series $y_i$.
Let us reorder the set of unknowns according to the value of their GK-dimension.
I.e. we suppose that $y_1,\ldots, y_{l_0}$ are polynomials that have zero GK-dimension;
$y_{l_0+1},\ldots,y_{l_1}$ have GK-dimension $\alpha_1$; \ \ldots \ ;
$y_{l_{r-1}+1},\ldots,y_{l_r}$ have GK-dimension $\alpha_r$,
where $0<\alpha_1<\ldots<\alpha_r$ and $l_r=N$.
For any given $s$ and for any given $i\in \{l_{s-1}+1,\ldots,l_s\}$
the $i$'th equation in System~\ref{eq::nonsym_system_eq} is of the following form:
\begin{equation}
\label{eq::induct_rat_rec}
y_i = \sum_{j= l_{s-1}+1}^{l_s} p_{j}(z) y_{j} + f_i(z,y_1,\ldots,y_{l_{s-1}}),
\end{equation}
where $p_j(z)$ are polynomials and $f_i$ is a polynomial in ($l_{s-1}+1$) variables
 with non-negative coefficients.
Namely, $f_i$ depends only on the first $l_{s-1}$
unknown variables and does not depend on $y_i$'s with $i>l_{s-1}$.
Indeed, Lemma~\ref{lem::sum_GK_dim} implies that the GK-dimension of the product  $y_{i_1}\ldots y_{i_k}$
is greater than $\alpha_s$
if there exists at least one $i_j>l_{s}$ or if there exists at least two different multiples $y_{i_j}$ and $y_{i_{j'}}$ with
$i_j>l_{s-1}$ and $i_{j'}>l_0$.

It remains to show that the solutions of the system~\eqref{eq::induct_rat_rec} are rational.
Actually, it is well known in the theory of generating functions that the solutions
of a linear system of equations with polynomial
coefficients are rational functions (see, e.g.,~\cite{Stanley_book}).
However, our system is generally not linear, so, we need
an induction argument.

By induction on $s$,  for the vector
${\bf y}_s = (y_{l_{s-1}+1}, \dots, y_{l_s})^T$
 we get the system of the form
$$
{\bf y}_s = A_s{\bf y}_s +B_s,
$$
where $A_s\in Mat_{(l_s-l_{s-1})\times (l_s-l_{s-1})}(z \Z [z])$ and $B_s$ is a vector of rational functions
which are equal to zero at the origin. The vector $B_s$ is obtained by substitution of solutions $y_i$'s with $i<l_{s-1}$ 
which are
rational by the induction hypothesis.
Then
$$
{\bf y}_s = (\Id - A_s)^{-1}B_s,
$$
so that all infinite series $y_i$ are rational functions.
Then the function $G_{\P}=\sum y_i$  is also rational.
\end{proof}

\section{Symmetric and shuffle operads}

\label{sec::sym_operads}
\label{sec::sym}

\subsection{Generating series for a shuffle composition}

\label{sec::operator_C}
\label{sec::sym_add_grading}

The first change we should do in the case of symmetric operads
(compared to what we have explained for non-symmetric operads)
is to change the type of generating series.
Suppose that a subset $\m\subset \b(\Upsilon)$ defines a monomial basis of a shuffle operad
$\p:=\f(\Upsilon)/(\Phi)$ (meaning that $\p$ is the quotient of the free operad
 $\f(\Upsilon)$ by the ideal generated by a subset $\Phi\subset\f(\Upsilon)$).
The exponential generating series of the dimensions of $\p$ is defined as follows:
\begin{gather*}
E_{\p} (z) := \sum_{n\geq 1} \dim \p(n) \frac{z^{n}}{n!}
 = \sum_{v\in \m} \frac{z^{ar(v)}}{ar(v)!}, \text{  where $ar(v)$ means the arity of $v$. }
\end{gather*}
If there is an additional grading of the set of generators $\Upsilon$
such that all relations from $\Phi$ are homogeneous, one can also
 consider an exponential  generating series in two variables.
Let $\m_n = \m \cap \p (n)$ be a homogeneous basis of $\p (n)$ and let $\m_{n,k}$ be the subset of
$\m_n$ consisting of the elements of degree $k$.
Then we define
 \begin{gather*}
E_{\p} (z,t) := 
\sum_{n\geq 1} \frac{z^n}{n!}
\sum_{k\in \Z} \#(\m_{n,k})  t^k =
\sum_{n\geq 1} \frac{z^n}{n!}
\sum_{m \in \m_{n}} t^{\deg m}.
\end{gather*}
One can equivalently define
$$
E_{\p} (z,t) = \sum_{n\geq 1} \frac{H_{\p(n)}(t)}{n!} z^n,
$$
where $H_{\p(n)}(t)$ is the Hilbert series of the graded vector space  $\p(n)$.
As in the non-symmetric case (Section~\ref{sec::nonsym_grading}),
we do not need the additional parameter $t$ in the most of our examples.
We provide our proofs mostly for the one-variable series; 
minimal modifications are need to prove analogous theorems for two-variable series.

Similar to the case of non-symmetric operads, one can define
\emph{a shuffle composition of vector spaces}
$\mu(\p_1,\ldots,\p_m)_{Sh}$ (where $\mu$ is an  element of a free shuffle operad $\F$ and
$\p_1,\ldots,\p_m$ are the graded vector subspaces of $\F$) as the vector space generated by all possible shuffle compositions
\begin{equation}
\label{eq:shuf_comp}
\mu(p_1,\ldots,p_m)_{Sh},
\end{equation}
where $p_i$ belongs to the graded component $\p_i(k_i)$
for all $i =1, \dots, m$ with  $k_1 + \dots+ k_m = n$.
Similarly to the non-symmetric case, each tree from the set of underlying
internally labeled trees in $\mu(\p_1,\ldots,\p_m)_{Sh}$
has $\mu$ as a root vertex and
the $i$-th subtree belongs to the basis of $\p_i$.
The main difference with non-symmetric operads concerns the external labeling.
As was mentioned in the definition of divisibility
in Section~\ref{sec::sym_background},
 the possible external labelings of a tree
from $\mu(\p_1,\ldots,\p_m)_{Sh}$  preserves the local order of minima of leaves in subtrees
(see the proof of Lemma~\ref{l:hilbert_shaf_com} below).

Let $R = \QQ [[z]]$ be the ring of formal power series.
Define a multilinear map $C: R^n \to R$ as follows:
\begin{gather}
\label{eq::C(f,g)}
                     C(f,g)(z) := \int_0^z f'(w) g(w) \, dw
                     \phantom{aaaaaaa}\mbox{ for $n=2$ } \\
        \text{ and \ } C(f_1, \dots , f_n) := C(f_1, C(f_2, \dots, f_{n}) )
         \phantom{aaaaaa}\mbox{for
         $n>2$}.
\label{eq::C(1..n)}
\end{gather}

The next Lemma establishes a connection between this operation and shuffle composition.

\begin{lemma}
\label{l:hilbert_shaf_com}
 Let $\mu, \p_1, \dots, \p_m$ be as above and let $S =
\mu(\p_1, \dots , \p_m)_{\sh}$. Then
$$
               E_S(z) = C(E_{\p_1}, \dots, E_{\p_m}).
$$
\end{lemma}

\begin{proof}
By linearity, it is sufficient to check the above relation in
the case of one-dimensional vector spaces $ \p_1, \dots, \p_m$.
 Assume that $\p_i$ is spanned by the basis element $p_i$ of arity $n_i$.
 Let $n = n_1 + \dots +n_m$.
 Then
$$
              E_S(z) = \frac{z^n}{n!} c(n_1, \dots, n_m),
$$
where $c = c(n_1, \dots, n_m)$ is equal to $\dim S(n)$. For each
$k =1, \dots, m $, denote by $N_k$ the $n_k$-element set
$\{ n_1 +\dots +n_{k-1} +1, \dots, n_1 +\dots +n_{k} \}$.
 It follows from the definition (cf.~\cite[Def.~2]{dk}) that
the number $ c(n_1, \dots, n_m)$ is equal to the number of
permutations  $\sigma \in \Sigma_n$ such that $\min \sigma(N_1) < \min
\sigma(N_2)< \dots <\min \sigma(N_m)$ and the restriction of
$\sigma $ to every $N_k$ is an isomorphism of ordered sets.
Therefore,  $ c(n_1, \dots, n_m)$ is equal to the number of
decompositions $[1..n] = Q_1 \cup \dots \cup Q_m$  with $|Q_k| =
n_k$ and $\min Q_1 < \dots <\min Q_m$ (here $Q_k = \sigma N_k$ for
some $\sigma$  as above). The first inequality is equivalent to
the condition $1 \in Q_1$, hence for every $Q_1 \ni 1$ (there are
$\binom{n-1}{n_1-1}$ ways to choose it) there is exactly
$c(n_2,\dots n_m )$ ways to get decompositions $Q_2 \cup \dots
\cup Q_m$ of the same kind for the set $[1..n]\setminus Q_1$.
Thus, we have the relations
$$
       c(n_1, n_2) = \binom{n_1+n_2-1}{n_1-1}
\qquad
\text{ and }
\qquad
       c(n_1, \dots ,n_m) = \binom{n-1}{n_1-1} c(n_2, \dots ,
       n_m).
$$
For the generating functions, we obtain the equalities
\begin{gather*}
          n \frac{z^{n_1+ n_2}}{n!} c(n_1, n_2) = \left( n_1 \frac{z^{n_1}}{n_1!} \right) \frac{z^{n_2}}{n_2!} ,
\\
          n \frac{z^n}{n!} c(n_1, \dots ,n_m) = \left( n_1 \frac{z^{n_1}}{n_1!} \right)
                      \left( \frac{z^{n_2+\dots n_m}}{(n_2+\dots n_m)!}  c(n_2, \dots ,
       n_m) \right),
\end{gather*}
which are equivalent to the desired integration equalities.
\end{proof}

\begin{rema}
\label{rem:shaf_comps_difur}
The  equation in Lemma~\ref{l:hilbert_shaf_com} is
equivalent to the following system of ordinary differential equations for
the functions $h_k(z) = E_{\mu(\p_{k}, \dots , \p_m)_{\sh}}(z)$:
$$
\left\{
 \begin{array}{l}
  h_1'(z) = E_{\p_1}'(z) h_2(z),\\
  h_2'(z) = E_{\p_2}'(z) h_3(z),\\
  \dots\\
  h_{m-1}' = E_{\p_{m-1}}'(z) E_{\p_m}(z)\\
 \end{array}
\right.
$$
with the initial conditions $h_k(0) = 0$ for $0 \le j \le m-k$.
This system uniquely determines the functions $h_1, \dots,
h_{m-1}$. 
\end{rema}

The following easy verified property of the operation $C$ will be used later
in Theorem~\ref{th::bosquet}.

\begin{prop}
\label{popr::C_1st_prop}
One has
$C(f,g)+C(g,f) = fg$
and, generally,
$$
\sum_{\sigma \in S_n}C(f_{\sigma (1)}, \dots , f_{\sigma (n)}) = f_1 f_2 \dots  f_n.
$$
In particular, $C(f,\ldots,f) = \frac{f^n}{n!}$.
\end{prop}

In view of Lemma~\ref{l:hilbert_shaf_com}, this means that
the sum of shuffle compositions of some vector spaces with respect to all  orderings is equal to their non-symmetric composition of the same arity.

\subsection{System of differential equations}

\label{sec::sym_main_theorem}

So far we were not able to formulate any statement about
generating series of an arbitrary shuffle operad with a finite
Gr\"obner basis. To establish some properties of these series, we
require additional assumptions, the main of which  is given in
Definition~\ref{def::shufl_regul} below. This assumption holds in a
number of examples, some of which are discussed below.  We have
checked also that for all known symmetric PBW operads there exists
a monomial shuffle operad with the same generating series as the
initial PBW operad but with the following property being satisfied
(see Conjecture~\ref{conj::PBW} below).

\begin{definition}\label{def::shufl_regul}
\label{def:shuffle_regular}
\begin{itemize}
\item The {\it planar skeleton} of a shuffle monomial $m$ in a free shuffle operad $\f$ is
the corresponding planar internally labeled tree, that is, it is
obtained from $m$ by erasing the labels (numbers) of all leaves.

\item A subset $\m$ of monomials in the free shuffle operad
$\f(\Upsilon)$ is called {\it shuffle regular} if for each
monomial $m\in\m$ all monomials with the same shuffle skeleton as $m$
belong to $\m$.

For example, the set
$$
\alpha(\beta(x_1,x_2),\gamma(x_3,x_4)),\alpha(\beta(x_1,x_3),\gamma(x_2,x_4)),
\alpha(\beta(x_1,x_4),\gamma(x_2,x_3))
$$
forms a shuffle regular subset with a shuffle skeleton $\alpha(\beta(,),\gamma(,))$.

\item A monomial operad $\p$ is {\it shuffle regular} if and only if the corresponding monomial basis
is a shuffle regular subset.

It is obvious that a monomial operad is shuffle regular if and only if the set of generating monomial
relations is shuffle regular.

\item Given a set of generators $\Upsilon$ of a  symmetric or shuffle
operad $\p$ and an admissible ordering of monomials, the operad
$\p$ is called {\it shuffle regular} if the set of leading terms
of the corresponding monomial ideal of relations is shuffle
regular. In other words, there exists a reduced Gr\"obner basis of
the ideal of relations of $\p$ in $\f(\Upsilon)$ with shuffle
regular set of leading terms.
\end{itemize}
\end{definition}

\begin{exam}
According to the Gr\"obner bases calculated in~\cite{dk}, one can
see that the operads {\it Com}, {\it AntiCom} and {\it Assoc} are
shuffle regular~\cite[Examples~8,10]{dk} whereas the operads {\it
Lie} and {\it PreLie} are not (with respect to given orders on
shuffle monomials)~\cite[Examples~9,11]{dk}. On the other hand, if
we change the ordering of monomials by  the dual
path-lexicographical ordering, the both operads {\it Lie} and {\it
PreLie} become shuffle regular.

\label{ex:reg_assoc} For instance, consider the shuffle operad of
associative algebras {\it Assoc}. The leading terms of a Gr\"obner
basis of the ideal of its relations are listed
in~\cite[Example~10]{dk}. They are the
shuffle monomials with the shuffle skeletons
$$
 \alpha(\alpha(a_1, -)-),  \alpha(\beta(a_1, -)-), \beta(\beta(a_1, -)-),
 $$
where  $\alpha:  \alpha(x,y) = x\cdot y$ and $\beta: \beta(x,y) =
y\cdot x$ are the generator operations for $\ass$. The cases of
other operads listed above are analogous.
\end{exam}

\begin{theorem}
\label{th:sym_mon_difurs}
 Let $\p$ be a shuffle regular symmetric operad  such that
 the corresponding set of generators and
 a Gr\"obner basis of relations are finite.
 Then  there exists an integer $N$
 and a system of integral equations on $N+1$ functions $y_0 = y_0(z,t),\ldots,y_N = y_N(z,t)$
\begin{equation}
\label{eq::sym_system_eq}
y_i =  t^{a_i}\sum_{s\in [0..N]^{d_i}} q^i_s
     C( y_{s_1}, \ldots,  y_{s_{d_i}})             \phantom{qqq} \mbox{ for }i =1, \dots, N  ,
\end{equation}
such that $E_{\p}(z,t) = \sum_{i=0}^{N}y_i(z,t)$, $y_0 = z$ and
$y_i(0,t)=y_i(z,0)=0$ for all $i>0$. The numbers $q^i_s\in
\{0,1\}$  and the nonnegative integers ${d_i},{a_i}$ and $
N $ are bounded from above by some functions of the degrees and
the numbers of generators and relations of the operad $\p$.
\end{theorem}

Our proof of Theorem~\ref{th:sym_mon_difurs}
(as well as the proof of Theorem~\ref{th::bosquet} below)
is  close to that of Theorem~\ref{th:nonsym_mon_difurs}.
The main difference is in the counting of the number of external labels of a planar tree.
This reduces to a simple change in the right-hand side
of the formula~(\ref{eq::nonsym_system_eq}):
$$
y_{s_1}\cdot\ldots\cdot y_{s_{d_i}}
                    \rightsquigarrow C(y_{s_1},\ldots,y_{s_{d_i}}).
$$
Namely, one should replace the product of the series by the sign
of the operator $C(\ldots)$ applied to them.
 In order to make our
exposition in symmetric case  self-contained we repeat one of
the proofs-algorithms in all details.

\begin{proof}[Proof of Theorem~\ref{th:sym_mon_difurs}]
Suppose that an operad $\p$ has a finite set of generators
$\Upsilon$ and a finite set of monomial relations $\Phi$. (It is
enough to consider the monomial relations since we are dealing
with generating series, therefore there is no difference between
the relations that form a Gr\"obner basis and the monomial
relations presented by the leading terms of the first ones.)
 Let $d$ be
the maximum level of leaves of elements in $\Phi$ (by the
\emph{level} of a vertex/leaf in a tree we mean the number of
vertices in a path from the root to this vertex/leaf). As was
mentioned in Proposition~\ref{lem::basis_in_sym}, every monomial
$v$ in a free operad $\f(\Upsilon)$ generated by $\Upsilon$  may
be identified with a rooted planar tree whose vertices are marked
by elements of $\Upsilon$ and whose leaves are numbered by natural
numbers $1,\ldots,ar(v)$ in such a manner that this numbering
preserves the ordering of minimums in each internal vertex. Given
such a monomial $v$, by its {\it stump} $b(v)$ we mean the shuffle
skeleton of its maximal monomial left divisor such that the leaves
and the internal vertices of $b(v)$ have levels strictly less than
$d$.

Let $\Stump$ be the set of all stumps of all nonzero monomials in $\p$.
Let $N$ be the cardinality of this set.
The elements $b_1, \dots, b_N$ of $\Stump$ are partially ordered by the
following  relation: $b_i < b_j$ iff $i\ne j$ and $b_i$ is a left
divisor of $b_j$ as a rooted planar tree.
 Let $\m_{b_i}$ be the set of all monomials in (= the
monomial basis of) the right-sided ideal generated by all possible
versions of the internal labeling of a stump $b_i$, and set
$$
\widetilde{\m_{b_i}} = \m_{b_i} \setminus \bigcup_{j: b_i<b_j} \m_{b_j}.
$$
The sets $\widetilde{\m_{b_i}}$ have empty pairwise intersections.
Moreover, the set $\bigcup_{i=1}^N \widetilde{\m_{b_i}} $ forms a
monomial basis of the operad $\p$. We have
$$
E_{\p}(z) = \sum_{i=1}^N y_i(z),
$$
where $y_i(z) = E_{span(\widetilde{\m_{b_i}})}(z)$ is the exponential generating series of the span
 of the set $\widetilde{\m_{b_i}}$. For every element (=operation)
$\mu \in \Upsilon$ of some arity $n$, define the numbers
$j_\mu(i_1, \dots, i_n)$ for all $1\le i_1, \dots, i_n\le N$ as
follows:
\begin{equation}
 \label{eq::stump_requr_index_sym}
       j_\mu(i_1, \dots, i_n) = \left\{
        \begin{array}{ll}
                   0, & \mbox{ if } \mu(\widetilde{\m_{b_{i_1}}}, \dots, \widetilde{\m_{b_{i_n}}})_{Sh} =0 \mbox{ in } \p
                   \\
                   j, & \mbox{ if the stump } b( \mu(\widetilde{\m_{b_{i_1}}}, \dots, \widetilde{\m_{b_{i_n}}})) =
                   b_j.
        \end{array}
                  \right.
\end{equation}
(where the shuffle compositions for monomial sets are defined as
the union of all compositions of type~(\ref{eq:shuf_comp})).
 Note that
the sets of the type $\mu(\widetilde{\m_{b_{i_1}}}, \dots,
\widetilde{\m_{b_{i_n}}})_{Sh}$ have vanishing pairwise intersections.

Let $v$ be a nonzero monomial in $\p$ with the root vertex labeled by $\mu$.
 Then $v \in \mu (\widetilde{\m_{b_{i_1}}}, \dots,
\widetilde{\m_{b_{i_n}}})_{Sh}$ for some $\mu, i_1, \dots, i_n$, that is,
$v = \mu(v_{i_1},\ldots,v_{i_n})_{\sigma}$
where the monomials  $v_{i_j} \in \widetilde{\m_{b_{i_j}}}$ and
a shuffle composition $\sigma$ are uniquely determined by
$v$.
Suppose that the index $j = j_\mu(i_1, \dots, i_n)$ from~(\ref{eq::stump_requr_index_sym})
is different from zero.
Then the shuffle regularity condition of a Gr\"obner basis
and the bound on the level of relations and stumps implies that  $v$  belongs to $\widetilde{\m_{b_j}}$.
We come up with
the following disjoint union decomposition
for all $j=1\ldots N$
\begin{equation}
\label{eq:tildes}
           \widetilde{\m_{b_{j}}} =  \bigcup_{j_\mu(i_1, \dots, i_n) = j} \mu(\widetilde{\m_{b_{i_1}}}, \dots,
                \widetilde{\m_{b_{i_n}}})_{Sh},
                \text{ where } \mu
                \text{ is a root vertex of each } v\in \widetilde{\m_{b_{j}}}.
\end{equation}
Then the equation~(\ref{eq::sym_system_eq}) corresponds to the
generating functions $y_i(z) = E_{span(\widetilde{\m_{b_i}})}(z)$
(where $a_j$ is the value of the corresponding grading on the
operation $\mu$). Similarly to the non-symmetric case, it follows
that $\m_{0}$ consists of the identity operation and all
$\widetilde{\m_{b_i}}$ contain elements of positive degrees
in generators. This implies the initial conditions on series
$y_{i}$.
\end{proof}


\begin{cor}
\label{cor:lin_syst} Let $\p$ be a finitely presented symmetric
operad with a finite shuffle regular Gr\"obner basis of relations.
Then there exists a system of ordinary differential
equations
\begin{equation}
\label{eq:dif_system_cor}
   y_i'(z)  +\sum_{j,l=1}^n q_{j,l}^i y_j y_l' = g_i(z)  
    \mbox{ for }i =1, \dots, n
\end{equation}
where $q_{j,l}^i \in \QQ$ and $g_i(z)\in \QQ[z]$ with the initial
conditions $y_1(0) = \dots=y_n(0)=0$, whose unique formal power
series solution $(y_1(z), \dots, y_n(z))$ satisfies the equality
$$
E_\p(z) =
y_1(z) +\dots +y_N(z)
$$
for some $N\le n$.
\end{cor}

\begin{proof}
 Let us introduce the functions $h_{i}$ as in
Remark~\ref{rem:shaf_comps_difur} for all combinations of the
power series $y_{i}$  which appear in the equations in the
statement of Theorem~\ref{th:sym_mon_difurs}.
Then the $i$-th equation given
in the statement of Theorem~\ref{th:sym_mon_difurs} is equivalent
to an equation  of the form
$$
y_i'(z) = \sum_{p=1}^n \sum_{s\in [1..N]^p} q^i_s y_{s_1}'(z)
h_{j(s)} + f_i'(z)
$$
(which is obtained by differentiation) with the initial condition
$y_i(0) = 0$. After the re-naming $y_{N+j} = h_j$ and adding
 the equations of the form
$$
y_{N+j}' = y_{k}' y_l
$$
with the same initial conditions $y_{N+j}(0) = 0$ (cf.
Remark~\ref{rem:shaf_comps_difur}), we obtain a system of
equations of the desired form which is equivalent (up to the ghost
variables $y_{N+j} = h_j$) to the initial system.
\end{proof}

\begin{rema}
The number of equations in the system of differential
equations~\eqref{eq:dif_system_cor} can in some cases be reduced. 
To do so, one can apply
to the shuffle regular monomial operads  the same methods as we
have discussed in Subsections~\ref{sec::few_relations}
and~\ref{sec::homology_algorithm} for the non-symmetric operads.
The first of these methods is illustrated in
Example~\ref{exam::NU3} below.
 We leave the detailed description of the algorithms to an interested reader.
\end{rema}

\begin{corollary}
\label{cor::main_cor}
The exponential generating series $E_{\p}(z)$ of a finitely presented operad $\p$ with a shuffle regular Gr\"obner basis
 is differential algebraic over $\QQ$.\footnote{If we consider power series
in two variables $t$ and $z$, then the coefficient ring $\QQ$ is replaced by the ring $\QQ[t]$ with the trivial differentiation
$\frac{d}{dz} t = 0$.}
That is, there exist a number $n\geq 0$ and a non-constant polynomial $\theta$ in $n+2$ variables
such that
              $$
                      \theta(z, E_\p(z), E_\p'(z), \dots, E_\p^{(n)}(z)) = 0.
              $$
\end{corollary}

\begin{proof}
By Artin's Approximation Theorem for differential
equations~\cite[Theorem~2.1]{dl}, for each positive integer $a$
there exists another power series solution $(\widetilde y_1(z),
\dots, \widetilde y_n(z))$ of the system~(\ref{eq:dif_system_cor})
such that all functions $\widetilde y_i$ are differential
algebraic and
$$
\widetilde y_i (z) = y_i(z) \quad  \mod \quad z^a
$$
for all $i = 1, \dots, n$.
Taking $a=1$ and using the fact that the solution with a zero constant term is unique,
we conclude that $\widetilde y_i (z) = y_i(z) $ for all $i$.
Since the sum of differential algebraic functions is again differential algebraic,
we conclude that $E_\p(z) =
y_1(z) +\dots +y_N(z)$ satisfies a differential algebraic equation.
\end{proof}


\subsection{Relation sets with additional symmetries}

\label{sec::sym_bosquets}

In the munber of cases a Gr\"obner bases of a shuffle regular operad has additional symmetries
which imply restrictions on the corresponging generating series
(Theorem~\ref{th::bosquet} and Corollary~\ref{cor::bosquet} below).

\begin{definition}
\label{def::sym_regular}
\begin{itemize}
\item The {\it tree skeleton} of a shuffle monomial $m$ in the free shuffle operad $\f$ is
the corresponding internally labeled rooted (non-planar) tree.
That is,  we erase the labels (numbers) of all leaves and forget
the planar planar representative of the  labeled tree $m$.

\item A subset $\m$ of monomials in the free shuffle operad $\f(V)$ is called {\it symmetric regular}
if for each monomial $m\in\m$ all monomials with the same tree skeleton belongs to $\m$.

For example, the set
$$
\left\{
\begin{array}{c}
\alpha(\beta(x_1,x_2),\gamma(x_3,x_4)),\alpha(\beta(x_1,x_3),\gamma(x_2,x_4)),
\alpha(\beta(x_1,x_4),\gamma(x_2,x_3)), \\
\alpha(\beta(x_2,x_3),\gamma(x_1,x_4)),\alpha(\beta(x_2,x_4),\gamma(x_1,x_3)),
\alpha(\beta(x_3,x_4),\gamma(x_1,x_2))
\end{array}
\right\}
$$
forms a symmetric regular subset with a tree skeleton
$\alpha(\beta(\textrm{-},\textrm{-}),\gamma(\textrm{-},\textrm{-}))$.

\item
The definitions of \emph{symmetric regular monomial operad} and
 arbitrary \emph{symmetric regular operad} are analogous to the ones
given in Definition~\ref{def::shufl_regul} for shuffle regular case.
\end{itemize}
\end{definition}

Obviously, the standard monomial basis of a symmetric regular
operad is again symmetric regular.

\begin{theorem}
\label{th::bosquet}
If the set of leading terms of a finite Gr\"obner basis of a shuffle regular operad $\p$ form a symmetric regular set
then the corresponding
system of recursive differential algebraic equations~\eqref{eq::sym_system_eq}
reduces to the system
of algebraic equations
\begin{equation}
\label{eq::bosquet:sym_system_eq}
y_i =  t^{a_i}\frac{1}{d_i !}\sum_{s\in [0..m]^{d_i}} q^i_s \cdot
      y_{s_1}\cdot \ldots\cdot  y_{s_{d_i}}             \phantom{qqq} \mbox{ for each }i =1, \dots, N
\end{equation}
for some formal power series $y_1, \dots, y_N$ with non-negative coefficients such that  $E_{\p}(z) = m_1 y_1 + \dots +m_n y_N $
for some integers $m_1,\ldots,m_n$.
\end{theorem}

\begin{proof}
Consider the algorithm given in the proof of
Theorem~\ref{th:sym_mon_difurs} applied to a symmetric regular
operad $\p$. Note that the generating series of
$\widetilde{\m_{b}}$ and $\widetilde{\m_{b'}}$ coincide if the
stumps $b$ and $b'$ have the same tree skeleton. Moreover, for each
collection of monomials $p_1,\ldots,p_n$ and each $n$-ary
operation $\mu$ and a permutation $\sigma\in S_n$ there is a
bijection between the tree skeletons of the elements of the set
$\mu(p_1,\ldots,p_n)_{Sh}$ and the tree skeletons of the elements of the set
$\mu(p_{\sigma(1)},\ldots,p_{\sigma(n)})$. Consider a relation
from the system~\eqref{eq::sym_system_eq} which corresponds to a
given stump $b_i$:
$$
 y_{i} = t^{a_{i}}\sum_{s\in [0..N]^{d}} q^i_s C( y_{s_1},\ldots,y_{s_d}).
 $$
Changing the subtrees of the root operation in a shuffle monomial
 one  may change the planar skeleton, whereas the tree skeleton remains the same.
Therefore,
$$
y_{i} = t^{a_{i}}\sum_{s\in [0..N]^{d}} q^i_s \frac{1}{d!} (\sum_{\sigma\in S_d} C( y_{s_{\sigma(1)}},\ldots,y_{s_{\sigma(d)}})) =
 \frac{t^{a_{i}}}{d!}\sum_{s\in [0..N]^{d}} q^i_s y_{s_1}\cdot\ldots\cdot y_{s_d}
$$
The last equality follows from Proposition~\ref{popr::C_1st_prop}.
Thus, the system~(\ref{eq::sym_system_eq}) of integration
relations can be replaced by the system of algebraic equations.
Moreover the equations are numbered by the appropriate subset of
tree-skeletons.
Again, these algebraic equations are numbered
by the tree skeletons of the monomials whose levels are less than
the maximal level of the relations.
\end{proof}

Theorem~\ref{th::bosquet}  is illustrated in Example~\ref{exam::NU3}.

Analogous to the case of non-symmetric operads, the classical
elimination theory implies the following
\begin{corollary}
\label{cor::bosquet}
 The exponential generating series $E_{\p}$ of a symmetric regular finitely presented operad $\p$
is an algebraic function.
\end{corollary}

\subsection{Operads of restricted growth}

\label{sec::sym_constant_growth}

We present here an application of the above theory to symmetric
and shuffle  operads of a restricted growth. We say that a
sequence $\{ a_n \}_{n\ge 0}$ of nonnegative real numbers has 
\emph{subfactorial} growth if for all positive constants $A,B>0$
there exists a constant $C>0$ such that $a_n < C
(\frac{n}{A})^{\frac{n}{B}}$ for all sufficiently large $n$. In
other words, this means that the growth
 $[ a_n ]$ of this sequence is less than the growth $[n!]$.
In particular, if the sequence is bounded by an exponent $C^n$
then its growth is subfactorial.

\begin{corollary}
\label{cor::sym_constant_growth}
 Let $\p$ be a symmetric or shuffle operad with a shuffle regular finite Gr\"obner basis.
Suppose that the growth of the sequence of dimensions $\dim
\p(n)$ is subfactorial. Then the exponential generating series
$E_{\p}$ satisfies a linear differential equation with
constant coefficients. Equivalently, the usual generating series
$G_{\p}=\sum_{n\geq 1} \dim\p(n) z^n$  is rational. In
particular, the sequence  $\dim\p(n)$ has 
exponential growth or polynomial growth with integer exponent.
\end{corollary}

The proof is similar to the proof of
Corollary~\ref{cor::nonsym_constant_growth}. The key point is to
reduce the system of equations~(\ref{eq::sym_system_eq}) for
the generating series to a system of linear recursive equations on
coefficients.

The series $y_b = \sum_{n\geq 0} y_{b,n} \frac{z^n}{n!}$
is an exponential generating series of the set of monomials in
$\widetilde{\M_{b}}$, hence the right hand side of any equation
from the system~(\ref{eq::sym_system_eq}) has nonnegative
coefficients. Therefore, for each stump $b$ the coefficient
$y_{b,n}$ of the exponential series $y_b$ is a nonnegative
integer bounded by the dimension $\P(n)$. Each equation is
numbered by an appropriate stump. However, the proof presented
below uses the type of the system rather than the combinatorics of
stumps. The statement of
Corollary~\ref{cor::sym_constant_growth} is still true in much
more general cases arising in some areas of combinatorics.

 There is a standard combinatorial data attached to the system~(\ref{eq::sym_system_eq}).
(See, e.g.,~\cite[p.~33]{anal_comb}). We say that a stump $b$
\emph{depends} on a stump $b'$ if the right hand side of the
corresponding recursive equation for generating series $y_{b}$
contains a nonzero summand of the form $C( y_{s_1}, \ldots,
y_{s_{m}})$, where some ${s_k}$ is equal to ${b'}$. In other words, the right
hand side of the equation~(\ref{eq::sym_system_eq}) corresponding
to $y_{b}$ depends on $y_{b'}$ in a nontrivial way. We say that
the dependence is {\em nonlinear} if the right hand side of the
recursive equation for $y_{b}$
 contains a multiple
$ C( y_{s_1}, \ldots,y_{s_k}, \ldots, y_{s_{m}})$,
 where $y_{s_k} = y_{b}$ and at least one of the  series $y_{s_j}$
 for $j\ne k$ is infinite.

Let us define a graph of dependencies for a system of recursive
equations~\eqref{eq::sym_system_eq}. It is  a directed graph with
vertices numbered by all possible stumps. A pair of stumps $b$ and
$b'$ is connected by an arrow if the stump $b$ depends on $b'$.
This graph is called the \emph{dependence graph} and will be
denoted by $\Gamma(\p)$. Whereas that dependence graph does not
contain all information about the system, it sometimes gives
growth conditions which we illustrate in the lemma below.

\begin{lemma}
\label{lem::edge_growth}
 Given an edge $b\rightarrow b'$ in the dependence graph $\Gamma(\p)$,
  there exists an integer $d$ and a polynomial $a(n)$
with positive integer values for all sufficiently large integers
$n$ such that the following inequality is satisfied for the
coefficients of the corresponding generating series:
\begin{equation}
y_{b,n} \geq a(n) y_{b',n-d}.
 \end{equation}
Moreover, if the dependence $b\rightarrow b'$ is nonlinear one may
choose the polynomial $a(n)$ to be nonconstant.
\end{lemma}
\begin{proof}
 By definition (see~\eqref{eq:tildes} for details),
  there is an arrow $b\rightarrow b'$ in the dependence graph
if and only if there exists a collection of stumps
$(b_1,\ldots,b_r)$ such that $b'= b_j$ for some $j$ and there is
an embedding of sets
$$\mu(\widetilde{\m_{b_1}},\ldots,\widetilde{\m_{b_r}})_{Sh} \subset \widetilde{\m_b}.$$
For each $i\neq j$ let us choose an element
$v_i\in\widetilde{\m_{b_i}}$.
 Let $d_i$ be the arity of the corresponding monomial $v_i\in\p$.
Consider the subset
$$
\m_{b\mapsto b'}:= \left\{
\begin{array}{c}
 \mu(w_1,\ldots,w_{j-1},\widetilde{\m_{b'}},w_{j+1},\ldots,w_r )_{Sh}  \\
\text{s.t. } w_i \text{ has the same shuffle skeleton as } v_i
\text{ for all }i\ne j
\end{array}
\right\}
\subset \mu(\widetilde{\m_{b_1}},\ldots,\widetilde{\m_{b_r}})_{Sh}.
$$
The number of elements of arity $n$ in the set $\m_{b\mapsto b'}$
is less than or equal to the number of elements of the same
arity in $\widetilde{\m_b}$. The detailed counting of the elements
in $\m_{b\mapsto b'}$ using Lemma~\ref{l:hilbert_shaf_com}
provides the  inequality
\begin{equation*}
\begin{split}
 \binom{n-1}{d_1-1}\times\ldots\times & \binom{n-\sum_{i\leq j-2} d_i -1}{d_{j-1}-1}
\times \binom{n- \sum_{i\leq j-1} d_i -1}{\sum_{i>j} d_i}\times  \\
& \times \binom{\sum_{i\geq j} d_i -1}{d_{j+1}-1}
\times \ldots \times \binom{d_r-1}{d_r-1} y_{b',n-\sum_{i\neq j} d_i}
\end{split}
\leq y_{b,n},
\end{equation*}
which shows the existence of $d$ and $a(n)$ as prescribed in the
lemma. Here each binomial coefficient in the left hand side  is
equal to number of shuffle monomials in the corresponding set
$\widetilde{\m_{b_i}}$ with shuffle skeleton $v_i$. Thus,  if the
dependence $b\rightarrow b'$ is nonlinear, then there exists a
collection of monomials $v_i$
such that the corresponding product of binomial coefficients is
different from constant.
\end{proof}

Lemma~\ref{lem::edge_growth} has a simple corollary for operads with small growth of dimensions:
\begin{corollary}
\label{cor::wheel_growth}
 If the growth of the dimensions $\p(n)$ is subfactorial then any loop
$b_1\rightarrow b_2 \rightarrow \ldots \rightarrow b_l \rightarrow b_1$
in the dependence graph $\Gamma(\p)$
 does not contain nonlinear dependencies.
\end{corollary}
\begin{proof}
Consider a collection of pairs $[(d_1,a_1(n)),\ldots,(d_l,a_l(n))]$
satisfying the following inequalities:
$$
y_{b_1,n} \geq a_1(n) y_{b_2,n-d_1} \geq a_1(n) a_2(n-d_1) y_{b_3,n-d} \geq \ldots
\geq \left( \prod_{j=1}^{l} a_j(n-\sum_{i=1}^{j-1}d_i) \right)
y_{b_1,n-(d_1+\ldots+d_n)}.
$$
Suppose that there is a nonlinear dependence in the given loop.
Then the degree of the polynomial $a(n):= \left( \prod_{j=1}^{l}
a_j(n-\sum_{i=1}^{j-1}d_i) \right)$ is positive, so that $a(n) > C
n^k$ for some $k\ge 1, C>0$ and for all  $n$ sufficiently large .
Therefore, $y_{b_1,n} \geq C n^k y_{b_1,n-d}
 \geq C^{m} n^k
(n-d)^k \ldots (n-(m+1)d)^k y_{b_1,n-md}$ for each $m\le n/d$.
In particular, if the series $y_{b_1}$ is different from zero there exists an infinite 
arithmetic progression of indices $\{n_0,n_0+d,n_0+2d,\ldots\}$ 
such that the corresponding sequence of coefficients $\{y_{b_1,n_0},y_{b_1,n+0+d},\ldots\}$ 
does not contain zeros.
 Put $m
=\lfloor\frac{n}{2d}\rfloor$. Then
$$
y_{b_1,n} \geq C^{\lfloor \frac{n}{2d}\rfloor}
\left(\frac{n}{2}\right)^{\lfloor \frac{n}{2d}\rfloor}
y_{b_1,n-{\lfloor \frac{n}{2d}\rfloor}d} \geq \left(\frac{n}{A}
\right)^{\frac{n}{B}}, \text{ if } y_{b_1,n-{\lfloor \frac{n}{2d}\rfloor}d}\neq 0
$$
where $A=2/C$ and $B={2d}+1$. Therefore, the only chance for $y_{b_1}$
to have a subfactorial growth is to have all polynomials $a_i(n)$
equal to positive constants. In particular,  all dependencies
should be linear.
\end{proof}

\begin{proof}[Proof of Corollary~\ref{cor::sym_constant_growth}]
The proof is by induction on the number of possible stumps or by
the number of vertices in the dependence graph $\Gamma(\p)$. The
induction base easily follows from
Corollary~\ref{cor::wheel_growth}
because any arrow in a graph with
one vertex is a loop.
 Therefore, the recursive equation on the
coefficients of generating series is linear.

\emph{Induction step.} Let $V$ be a maximal proper subset of
vertices in the graph $\Gamma(\p)$ such that there is no outgoing
arrows to the remaining set of vertices $\bar{V}$. 
Let $G$, $\bar{G}$ be a maximal subgraph spanned by $V$, 
respectively $\bar{V}$ and all arrows between them.
(We omit arrows coming from $\bar{V}$ to $V$.) Notice that the
subgraph $G$ may be empty and, on the contrary, $\bar{G}$ 
contains at least one vertex. Moreover, the subgraph $G$ is
a dependence graph for the subset of stumps numbered by vertices
in $V$. From the induction hypothesis it follows that for each
$b\in V$ the corresponding usual generating series $G_{b}=
\sum_{n\geq 0} y_{b,n} z^n$ is rational.
Lemma~\ref{lem::exp_to_rational} given below shows that any
summand of the form $C(y_{b_1},\ldots,y_{b_l})$ where all $b_i$
belong to $V$ is an exponential generating series of a
sequence such that the corresponding ordinary generating function
is  rational. On the other hand, the maximality property of $V$
implies that each two vertices from $\bar{V}$ are connected
by a directed path. Therefore, any arrow $b\rightarrow b' \subset
\bar{G}$ belongs to a loop where the remaining part of the
loop is a directed path from $b'$ to $b$.
Corollary~\ref{cor::wheel_growth} implies that all dependencies in
this wheel are linear. Hence, the system for the usual
generating series $G_{b}$ with $b\in \bar {V}$ reduces to a system
of linear equations with rational coefficients. As we have seen at
the end of Section~\ref{sec::nonsum_const_growth}, this implies
that all ordinary generating series are rational.
\end{proof}

The following Lemma is well known. We include its simple proof for
completeness.

\begin{lemma}
\label{lem:rat=dif-lin}
Let $G_a(z) = \sum_{i\ge 1} a_n z^n $ and
$E_a(z) = \sum_{i\ge 1} \frac{a_n}{n!} z^n $ be the ordinary and
exponential generating functions of the same sequence of complex
numbers $\{a_n\}_{n\ge 1}$. Then the function $G_a(z)$ is rational
if and only if the function $E_a(z)$ satisfies a non-trivial linear
differential equation with scalar coefficients.
\end{lemma}

\begin{proof}
The condition ``$G_a(z)$ is rational'' means that a recurrent
equation
$$
 a_{n+k} = \sum_{j=0}^{k-1} c_j a_{n+j}
$$
holds for all $n\ge 1$. It is equivalent to the recurrent relation
$$
 \frac{(n+k)!}{n!} b_{n+k} = \sum_{j=0}^{k-1} c_j \frac{(n+j)!}{n!} b_{n+j}
$$
for the numbers $b_n = a_n /n!$. This is equivalent to the
differential relation
$$
E_a(z)^{(k)} = \sum_{j=0}^{k-1} c_j E_a(z)^{(j)}
$$
for the exponential generating function $E_a(z)$.
\end{proof}

Let $\mathcal{E}$ be the set of all exponential generating series
such that the corresponding ordinary generating series are
rational functions. In other words, $\mathcal{E}$ is the set of
all exponential generating functions which are  solutions  of
non-trivial linear differential equations with scalar
coefficients. Now, the next Lemma is obvious.

\begin{lemma}
\label{lem::exp_to_rational} The set $\mathcal{E}$ is closed under
multiplication, differentiation and integration. In particular if
$E_f,E_g\in \mathcal{E}$ then $C(E_f,E_g) = \int_{0}^{z} E_f'(w)
E_g(w) dw$ also belongs to $\mathcal{E}$, that is,  the
corresponding ordinary generating series is also rational.
\end{lemma}

\subsection{Examples for symmetric operads}
\label{sec::examples:sym}
Unfortunately, the list of known operads is not so big and most of examples where a Gr\"obner basis has been computed
are quadratic operads. (See~\cite{zin} for an incompleted list of quadratic operads.) 
Most of them admit an ordering such that the corresponding Gr\"obner basis is shuffle regular.
We present below a couple of examples which illustrate the theory and possible orderings.

\subsubsection{Examples of shuffle regular PBW operads}

\begin{example}
\label{ex:alia}
Consider the class of so-called alia algebras introduced by Dzhumadil'daev~\cite{dzh},
that is, the algebras with one binary operation (multiplication) and satisfying the identity
$$
\{[x_1, x_2], x_3\} + \{[x_2, x_3], x_1\} + \{[x_3, x_1], x_2\} = 0,
$$
where $[x_1,x_2] = x_1x_2-x_2x_1$ and $\{x_1,x_2\} = x_1x_2+x_2x_1$ (these algebras are also referred in~\cite{dzh}
as 1-alia algebras).

Let us choose the generators $\alpha: (x_1,x_2)\mapsto [x_1,x_2]$ and $\beta: (x_1,x_2)\mapsto \{x_1,x_2\}$
for the corresponding symmetric/shuffle operad $\alia$.
Then the leading term of the relation corresponding to the above identity (with respect to the path-lex order with $\beta>\alpha$)
 is $\beta(x_1,\alpha(x_2, x_3))$, that is, the only shuffle monomial corresponding to the shuffle skeleton $\beta(-,\alpha(-, -))$.

Obviously,  there is no overlapping of the leading term
$\beta(x_1,\alpha(x_2, x_3))$ of the relation with itself, hence
the relation is the unique element of the Gr\"obner basis of the
relations of $\alia$. Then we have the following three elements of
the set of stumps $B$ (in terms of the proof of
Theorem~\ref{th:sym_mon_difurs})):
$$
B_{0} = \Id, B_1 = \alpha, B_2 = \beta.
$$
We get the relations
$$
\left\{
\begin{array}{l}
y_0 = z,\\
y_1 = C(E_{\alia},E_{\alia}),\\
y_2 = C(E_{\alia}, y_0+y_2),\\
E_{\alia} = y_0+y_1+y_2
\end{array}
\right.
$$
(here we use the linearity of the operation $C$ to shorten the summation in the right hand sides).
Using the linearity of $C$ and Proposition~\ref{popr::C_1st_prop}, we get
the system
$$
\left\{
\begin{array}{l}
y_0 = z,\\
y_1 = E_{\alia}^2/2,\\
y_2 = E_{\alia}^2/2 - C(E_{\alia},y_1),\\
E_{\alia} = y_0+y_1+y_2,
\end{array}
\right.
$$
which leads to the equation
$$
    C(y, y^2/2 ) =z-y+y^2
$$
 for $y = E_{\alia}(z)$.
After differentiation, we get
the equation
$$
    y' y^2/2 = 1-y'+2yy'.
$$
Using the initial condition $y(0) = 0$, we get
the algebraic equation
$$
  y^3/6-y^2+y = z
$$
for $y$. In particular, the  function $y = E_{\alia}(z)$ is algebraic.

Since the operad $\alia$ has a quadratic Gr\"obner basis of relations,
it follows from~\cite{hof} (see also \cite[Cor.~3]{dk}) that the operad $\alia$ is Koszul.
By the Ginzburg--Kapranov relation, its exponential generating series $y$ satisfies the relation
$$
   f(-y) = -z,
$$
where $f(z)$ is the exponential generating series of the quadratic dual operad $\alia^!$.  It follows from the
equation above that
$$E_{\alia^!}(z) = z+ z^2+z^3/6.
$$
This polynomial coincides with the result of a direct calculation
given in the concluding remark of~\cite{dzh}.

Recall now two other classes of algebras from~\cite{dzh}. A
nonassociative algebra is called {\em left}  (respectively, {\em
right}) alia
 if it satisfies the identity
$$
l(x_1,x_2, x_3) = [x_1, x_2]x_3 + [x_2, x_3]x_1 + [x_3, x_1]x_2 = 0,
$$
or, respectively, the identity
$$
r(x_1,x_2, x_3) = x_1[x_2, x_3] + x_2[x_3, x_1] + x_3[x_1, x_2] = 0.
$$
Consider the left hand side
$
r(x_1,x_2, x_3)
$
of the last identity.
Using the above generators $\alpha$ and $\beta$ with the substitution $2x_1x_2 = \alpha(x_1,x_2)+\beta(x_1,x_2)$,
we see that the leading monomial of $r(x_1,x_2, x_3)$
with respect to the same path-lex order with $\beta>\alpha$
is the same monomial $\beta(x_1,\alpha(x_2, x_3))$ as for alia algebras.
By the same reasons as  above, we see that the operad of right alia algebras is PBW and Koszul
with the same generating series as the operad $\alia$.
By the right-left symmetry, the same is true for left alia algebras.
Thus, we get
\begin{prop}
The three operads for alia algebras, left alia algebras and right alia algebras
are Koszul with the same  exponential generating series $ y = E_{\p}(z)$ satisfying the equation
$$
y^3/6-y^2+y = z,
$$
that is,
$$
y(z) = z+z^2+\frac{11}{6}z^3+\frac{25}{6}z^4+\frac{127}{12}z^5+
\frac{259}{9}z^6+\frac{1475}{18}z^7+\frac{17369}{72}z^8+\frac{943855}{1296}z^9+\frac{2906189}{1296}z^{10}+O(z^{11}).
$$
Each of their three Koszul dual operads is finite-dimensional and has exponential
generating series
$$
E_{\p^!}(z) = z+z^2+z^3/6.
$$
\end{prop}
\end{example}

\subsubsection{Examples of symmetric regular operads}

In the next two examples, we consider the operad of upper triangular matrices over non-associative commutative  rings.

\begin{example}
\label{ex::2x2-upper-matrices}
Let $R$ be a commutative non-associative ring (or a $\kk$-algebra).
Then it is easy to
see that  the algebra $UT_2(R)$ of upper triangular $2\times2$-matrices
 over $R$ satisfies the identity
$$
   [x_1,x_2][x_3,x_4] = 0,
$$
where  $[a,b] = ab-ba$.

Let us denote by  ${{NU}}_2$ the operad generated by the operation
of non-symmetric multiplication $\mu: (x_1,x_2)\mapsto x_1x_2$
(i.~e., the arity two component ${{NU}}_2$ is spanned by $\mu$ and
$\mu': (x_1,x_2)\mapsto x_2x_1$) subject to this identity.
Consider the corresponding shuffle operad ${\mathcal {NU}}_2$
generated by two binary generators, namely,
 the operations $\mu$ and
$\alpha: (x_1,x_2) \mapsto [x_1,x_2]$. Then the above identity is
equivalent to the pair of shuffle regular monomial identities
$$
f_1 = \mu(\alpha(\textrm{-},\textrm{-}), \alpha(\textrm{-},\textrm{-}) )=0
\quad \text{ and } \quad
f_2 = \alpha(\alpha(\textrm{-},\textrm{-}), \alpha(\textrm{-},\textrm{-}) )=0.
$$
Therefore, the ideal of relations of the shuffle operad ${\mathcal {NU}}_2$
is generated by the following six shuffle monomials obtained from $f_1$ and
$f_2$ by substituting all shuffle compositions of four variables
(which we denote for simplicity by 1,2,3,4):
$$
\begin{array}{lll}
m_1 = \mu(\alpha(1,2), \alpha(3,4) ),
 & m_2 = \mu(\alpha(1,3),
\alpha(2,4) ),
  & m_3 = \mu(\alpha(1,4), \alpha(2,3) ),\\
m_4 = \alpha(\alpha(1,2), \alpha(3,4) ),
 & m_5 = \alpha(\alpha(1,3),
\alpha(2,4) ),
  & m_6 = \alpha(\alpha(1,4), \alpha(2,3) ).\\
\end{array}
$$
Let us describe the set $B$ of all stumps of all nonzero monomials in
$\NU_2$. Since the relations have their leaves at level 2, $B$
includes all monomials of level at most one, that is, the
monomials
$$
B_0 = \Id, B_1 = \mu(\textrm{-},\textrm{-}), B_2 =\alpha(\textrm{-},\textrm{-}).
$$
For the corresponding generating series $y_i = y_i(z)$ with
$i=0,1,2$ we have
$$
\left\{
 \begin{array}{l} y_0 = z, \\
  y_1 = C(y_0,y_0) + C(y_1,z)+C(z,y_1) +C(y_2,z)+C(z,y_2) +
  C(y_1,y_1)+   C(y_1,y_2)+  C(y_2,y_1), \\
  y_2 = C(y_0,y_0) + C(y_1,z)+C(z,y_1) +C(y_2,z)+C(z,y_2) +
  C(y_1,y_1)+   C(y_1,y_2)+  C(y_2,y_1).
  \end{array}
 \right.
$$
We see that $y_1(z) = y_2(z)$  and $E_{\NU_2}(z) = y(z) =
y_0(z)+y_1(z)+y_2(z) = z+2y_1(z)$. The second equation of the
above system gives, after differentiation, the ordinary
differential equation (ODE)
$$
y_1' = z  + 2zy_1+2zy_1' + 3y_1y_1'
$$
with the initial condition $y_1(0)=0$, which is equivalent to the
ODE
$$
(y'(z)-1)(2-z-3y(z)) = 4y(z)
$$
on $y(z)$, again with the initial condition $y(0) = 0$. It follows
that
$$
\begin{array}{l}
E_{{\mathcal {NU}}_2}(z) = y(z) = \frac{1}{3}\left( 2-z-2 \sqrt {1-4\,z+{z}^{2}} \right) \\
= z+{z}^{2}+2\,{z}^{3}+{\frac {19}{4}}{z}^{4}+{\frac
{25}{2}}{z}^{5}+{ \frac {281}{8}}{z}^{6}+{\frac
{413}{4}}{z}^{7}+{\frac {20071}{64}}{z}^ {8}+{\frac
{31249}{32}}{z}^{9}+{\frac {396887}{128}}{z}^{10} + o(z^{10})
 \end{array}
 $$
\end{example}

Let us generalize Example~\ref{ex::2x2-upper-matrices} to the case
of matrices of order $n$. The following description of identities
easily follows from the fact that the diagonal elements of the
commutator of two upper triangular matrices are zero.

\begin{lemma}
\label{lem:upper-triangle-define}
Let  $U_n(R)$ be the algebra of upper triangular matrices of order $n$ over a (non-associative)
commutative ring $R$.
Then for each $n$-ary multiple composition $f$ of multiplications of matrices,
the  identity of $2n$ arguments
$$
   f([x_1, x_2], \dots, [x_{n-1}, x_n]) = 0,
$$
holds in $U_n(R)$, where, as usual,  $[a,b]$ stands for $ab-ba$.
\end{lemma}

For example, for $n=2$ we get the single identity $[x_1,x_2][x_3,
x_4]$ as above. For $n=3$  we have the identities $f_i([x_1, x_2],
[x_3, x_4], [x_{5}, x_6])$ with $i=1,2$, where $f_1(a_1,a_2,a_3) =
(a_1a_2)a_3$ and  $f_2(a_1,a_2,a_3) = a_1(a_2a_3)$ (all other
identities are obtained from these two by permutations of
variables).

Consider the operad ${\mathcal {NU}}_n$ of upper triangular
matrices generated by the non-symmetric operation of
multiplication $\mu: (x_1,x_2)\mapsto x_1x_2$ subject to all these
identities. Consider natural generators of the corresponding
shuffle operad $\alpha : (x_1,x_2)\mapsto [x_1,x_2]$ and $\beta:
(x_1,x_2) \mapsto x_1x_2+x_2x_1$. We immediately see that the set
of shuffle relations of this operad is spanned by monomials and is
symmetric regular.

\begin{cor}
For all $n\ge 2$ the exponential generating series $ E_{{\mathcal {NU}}_n}$ is algebraic.
\end{cor}

\begin{exam}
\label{exam::NU3}
Let us find a relation for the
exponential generating series for the operad
$\p={\mathcal {NU}}_3$. To do this, we use the appropriate version of the method used in Subsection~\ref{sec::few_relations}.

The minimal set of the monomial relations
of this operad $\p={\mathcal {NU}}_3$ consists of the monomials with one of the following 4 tree skeletons:
$$
 \xi(\zeta(\alpha(\textrm{-},\textrm{-}),\alpha(\textrm{-},\textrm{-})), \alpha(\textrm{-},\textrm{-})),
\text{ where } \xi,\zeta \in \{\alpha,\beta\}.
$$
Then the set $T(\p)$ from  Lemma~\ref{lem::buts_few_relations} consists of the following five tree skeletons:
$$
I  \text{ (the identical operation)}, a := \alpha(\textrm{-},\textrm{-}), b := \beta(\textrm{-},\textrm{-}), A := \alpha(\alpha(\textrm{-},\textrm{-}), \alpha(\textrm{-},\textrm{-})),
B:= \beta(\alpha(\textrm{-},\textrm{-}), \alpha(\textrm{-},\textrm{-})).
$$
We get the following equations for the corresponding generating functions $ y_{a}$,  $y_{b}, y_A, y_B$, and $y_I = E_{\p}$:
$$
\left\{
\begin{array}{l}
y_I = z+ y_{a} + y_b,\\
y_{a} = y_b = \frac{1}{2}y_I^2 - \frac{1}{2}( y_A y_a + y_B y_a + y_a y_A +y_a y_B) +\frac{1}{2} y_A^2, \\
y_A = y_B = \frac{1}{2} y_a^2 - \frac{1}{2}(y_A y_a + y_a y_A) + \frac{1}{2} y_A^2.
\end{array}
\right.
\Leftrightarrow
\left\{
\begin{array}{l}
y_I = z+ 2y_{a},\\
y_{a} = \frac{1}{2}y_I^2 - 2 y_a y_A +\frac{1}{2} y_A^2, \\
y_A =  \frac{1}{2} y_a^2 - y_a y_A + \frac{1}{2} y_A^2.
\end{array}
\right.
$$
After elimination of the variables $y_A$ and $y_a$, we obtain the following algebraic equation for $y_I = E_{{\mathcal {NU}}_3}$:
$$
{y_I}^{4}+ \left( 12\,z-24 \right) {y_I}^{3}+ \left( 30\,{z}^{2}+8\,z+80 \right) {y_I}^{2}+ \left(-36\,{z}^{3}+ 24
\,{z}^{2}-32\,z -64 \right) y_I+9\,{z}^{4
}-8\,{z}^{3}+16\,{z}^{2}+64\,z  = 0.
$$
It follows that
$$
E_{{\mathcal {NU}}_3} = z+{z}^{2}+2\,{z}^{3}+5\,{z}^{4}+14\,{z}^{5}+{\frac {167}{4}}{z}^{6}+
130\,{z}^{7}+{\frac {26745}{64}}{z}^{8}+{\frac {44045}{32}}{z}^{9}+{
\frac {36969}{8}}{z}^{10}+O \left( {z}^{11} \right) .
$$
\end{exam}

\subsubsection{Examples of computations via homology}

\begin{exam}
\label{ex::sym_homol_exam}
Let us define a class of operads where all relations are relations on the commutators.

Namely, any given finite set of operations $\Upsilon$ and
finite set of linearly independent elements $\Phi\subset\f(\Upsilon)$
defines an operad
$$
\q_\Phi := \f(\Upsilon\cup \{[\textrm{-},\textrm{-}]\})  /
\left(
\begin{array}{c}
[[x_1,x_2],x_3] +[[x_2,x_3],x_1] +[x_3,x_1],x_2] = 0  \\
 g([\textrm{-},\textrm{-}],\ldots,[\textrm{-},\textrm{-}]) \text{ for all } g\in \Phi
\end{array}
\right)
$$
In other words, the operad $\q_{\Phi}$ is generated by the union of the set $\Upsilon$ and a Lie bracket $[\textrm{-},\textrm{-}]$.
The first identity is the Jacobi identity for the Lie bracket.
 The set of additional identities in the lower line consists of substitutions of the Lie bracket into
the arguments
of the operations in $\Phi$. (I.~e., we put the bracket $[\textrm{-},\textrm{-}]$ in each leaf of $g$).

\begin{proposition}
\label{prop::hom_exam}
 The above set of  relations of the operad $\q_\Phi$  forms
a shuffle regular Gr\"obner basis
with respect to the dual path-lexicographical ordering.
(The ordering of generators does not matter.)
\end{proposition}

\begin{proof}
There are two types of nontrivial intersections of leading terms of relations in $\q_\Phi$:
\begin{itemize}
 \item The leading term $[x_1,[x_2,x_3]]$ of the Jacobi identity intersects with
 itself.
The corresponding $s$-polynomial written for the monomial $[x_1,[x_2,[x_3,x_4]]]$
reduces to zero because the operad Lie is PBW according to the dual path-lexicographical ordering.
 \item For any given $g\in\Phi$
the leading term $\widehat g([,],\ldots,[,])$ of the relation $g([,],\ldots,[,])$ of arity $m$
has $m$ different intersections with the leading term of the Jacobi identity.
\end{itemize}
The typical $s$-polynomial of the second kind looks as follows:
$$
S:= g(\ldots,[x_i,[x_j,x_k]],\ldots) - \hat{g}(\ldots,J(x_i,x_j,x_k),\ldots),$$
 where  $J(x_i,x_j,x_k):= [x_i,[x_j,x_k]] -[[x_i,x_j],x_k]+[[x_i,x_k],x_j]$ is a notation for Jacobiator.
The following reductions reduces this $s$-polynomial to zero:
$$
S + g(\ldots,[[x_i,x_j],x_k],\ldots) - g(\ldots,[[x_i,x_k],x_j],\ldots)
 -(g-\hat{g})(\ldots,J(x_i,x_j,x_k),\ldots) = 0.
$$
\end{proof}
Starting from a Gr\"obner basis one can define a resolution (also
called a DG  model) for the monomial replacement as it is
prescribed in~\cite{dk_resolutions} and then deform the
differential to get a resolution for $\q_\Phi$. In order to
compute the generating series of the operad it is enough to
describe the generators in the last resolution. This is given in
the following
\begin{proposition}
 There exists a resolution of $\q_\Phi$  generated by the union of the following sets:
\begin{itemize}
 \item the set of generators $\Upsilon$ with homological degree zero,
 \item the set of right-normalized commutators $L:=\{[x_1,[x_2,\ldots,[x_{k-1},x_{k}]\ldots]]|k\geq 2\}$,
 (The homological degree of the commutator on $k$-letters is equal to $(k-2)$.)
 \item the union of sets $\{ g(L,\ldots,L)_{Sh} \}$ for all $g\in \Phi$
(The homological degree of $g(l_1,\ldots,l_k)_{Sh}$ is equal to $\sum \deg (l_i) +1$.)
\end{itemize}
\end{proposition}
The generating series of the right normalized commutators $L$ is
$( e^{-z}+z-1)$. Multiplying the generating series of each of the
above sets by $(-1)^h$, where $h$ is the homological degree, we
get the equation \begin{equation} \label{eq::series::Lie::Phi}
E_{\q_{\Phi}}^{-1}(z)  =1 -  e^{-z} - E_{\Upsilon}(z) - E_{\Phi}(
e^{-z}+z-1)
\end{equation}
for  the functional inverse $E_{\q_{\Phi}}^{-1}$ of the
exponential generating series of the operad $\q_{\Phi}$.
\end{exam}

\begin{exam}
\label{ex::hom_exam_3}
 Consider a particular example of the above construction consisting of
the class of Lie-admissible non-associative algebras. Let $\p$ be
an operad generated by one binary non-symmetric operation
(multiplication) subject to the following identities:
\begin{gather*}
[a,[b,c]]+[b,[c,a]]+[c,[a,b]] = 0 , \\
[a,b][c,d] +[c,d][a,b]=0
\end{gather*}
where, as usual,  $[a,b] = ab-ba$. The operad $\p$  satisfies
the conditions of Proposition~\ref{prop::hom_exam}
 with $\Upsilon = \Phi = \{ \beta: (x_1,x_2)\mapsto x_1x_2+x_2x_1 \}$.
It follows that $E_{\Upsilon}(z) = E_{\Phi}(z) = z^2/2$. The
identity~\eqref{eq::series::Lie::Phi} gives the following equation
on the generating series $E_{\p}(z)$ after appropriate
substitutions:
$$
 E_{\p}^{-1}(z) = 1-  e^{-z} -\frac{z^2}{2} - \frac{( e^{-z} +z-1)^2}{2}
\Leftrightarrow 3-2\,{ E_{\p}}+2\,{  E_{\p}}\,{ e^{{-  E_{\p}}}}+{ e^{{-2  E_{\p}}}}-4\,{ e^{- E_{\p}}}
= 2z.
$$
It follows that
$$
E_{\p} (z) = z+{z}^{2}+{\frac {11}{6}}{z}^{3}+{\frac {
49}{12}}{z}^{4}+{\frac {1219}{120}}{z}^{5}+{\frac {811}{30}}{z}^{6}+{
\frac {75919}{1008}}{z}^{7}+{\frac {97175}{448}}{z}^{8}+{\frac {
25827439}{40320}}{z}^{9}+{\frac {116679221}{60480}}{z}^{10}+O \left( {
z}^{11} \right).
$$

\end{exam}

\section{Remarks and conjectures}

\label{sec::final_remarks}

\subsection{An analogy with graded associative algebras}
\label{subs:motiv}

Many results in the theory of operads have an analogy in the
theory of graded associative algebras. We now  review some particular cases.

 For a graded finitely generated associative
algebra $A = A_0\oplus A_1 \oplus A_2 \oplus \dots$ over a field
$\kk$, its {\em Hilbert series} is defined as the generating function
for the dimensions of the graded homogeneous components:
$$
               H_A(z) = \sum_{n\ge 0} (\dim_\kk A_n) z^n.
$$
Whereas in general the series $H_A(z)$ can be very complicated
(e.g.,  a   transcendental function), for a number of important classes of
algebras it is a rational function
$$
              \frac{p(z)}{q(z)}, \mbox{ where } p(z), q(z)  \in
              \Z[z].
$$
Indeed, $H_A(z)$ is rational if  $A$ is
 commutative, or Noetherian with a polynomial identity, or  relatively free, or Koszul (according to a conjecture of~\cite{pp}),
 etc.
 A general class of algebras with
 rational Hilbert series is the class of algebras with
 a finite Gr\"obner basis of  relations.
 Finitely presented monomial algebras, PBW algebras,
 and many other common types of algebras are particular examples of this class.
 For a survey of these and
  other results on Hilbert series of associative algebras,
 we refer the reader to~\cite{ufn} and references therein.

We infomally conjecture that as each common algebra has a rational Hilbert series, 
for  each common operad the generating function will also belong to  a particular  identifiable set $M$.
Let us try to bound this hypothetical set $M$.

Certainly,  the generating series of the most useful operads such
as operads of commutative, Lie, and associative algebras should
belong to $M$. Since
we have
$$
E_{\com} =  e^z-1, E_{\ass} = \frac{z}{1-z}, E_{\lie} = -\ln(1-z),
$$
we conclude that, at least, the
generating series of a generic symmetric operad
may be non-rational and may be exponential or logarithmic.
In addition, finite dimensional operads seem
to be simple enough to have ``general'' generating series.
Hence $M$ should also include the polynomials with
rational coefficients.
Free finitely generated operads are also
general; if the generating series of the vector space (more
precisely, the symmetric module) $V$ of generators of such an operad
$\F$ is a polynomial $p(z) = E_V$, then the generating series of
$\F$ is
$$
E_{\F} = (f_V^{-1})(z),
$$
where $f_V(z) = z - p(z)$ and $\cdot^{-1}$ stands for the inverse function.
 This means that $M$ also
include algebraic (over $\QQ$) functions.

Note that  the sets of quadratic,
finitely presented, and the other main types of operads are closed
under direct sum with
common identity component and composition. Thus,
it seems reasonable to assume that $M$ is closed under the corresponding
formal power series operations, that is, the operations that send
the pair of formal power series $f(z) = E_{\p}(z)$ and $g(z) =
E_{\q}(z)$ to
$$
E_{\p \oplus \q} =  f(z) + g(z) - z
$$
and
$$
E_{\p(\q)} = f(g(z)).
$$

 Finally, since the Koszul duality
plays an important role in the theory of operads and its applications,
one would want to have that if a generating series of a Koszul operad $\P$
belongs to $M$, then the series of its Koszul dual $\P^!$
should also belong to $M$. To ensure this, we assume that $M$ is closed under the
operation which sends $f(z)$ to
$$
   -(f^{-1})(-z).
$$

Note that while the generating series of $\com$ and $\lie$ are neither
rational nor algebraic, both these series satisfy simple first-order
differential equations. In particular, these functions
 are differential algebraic. In view of the
consideration above, we state the following claim.

\begin{thesis}
\label{main_thesys}
 The exponential generating series of a generic finitely
presented symmetric operad is differential algebraic.
\end{thesis}

As for the ordinary  generating series of non-symmetric operads (i.e.,
the exponential generating series of symmetric operads which are
symmetrizations of the non-symmetric ones), the class of
such general functions should include, at least,  the above generating
function of $\ass$, polynomials over nonnegative integers (for
finite dimensional operads) and formal power series which are
roots of polynomials with integer coefficients (for free operads). In
particular, such generating series are algebraic, that is, they
satisfy some non-trivial algebraic equations over $\mathbb{Z}[z]$.

Some particular examples of differential equations for generating series of symmetric operads 
is collected in~\cite{zin} and by Chapoton in~\cite{Chapot_encycl}, as he kindly pointed out to us.

\begin{thesis}
\label{nonsym_thesys}
 The generating series of a generic finitely
presented non-symmetric operads is algebraic over $\mathbb{Z}[z]$.
\end{thesis}

Suppose that the sequence of dimensions  $\dim \p_n$ of the components of an operad
has slow growth (e.g., $\dim \p_n$ is  bounded by some polynomial of $n$). Then in many  examples
the ordinary generating series $G_{\p}$ of
the operad $\p$ is quite simple. For example, the operad $\com $
has a rational generating series:
$$
G_{\com} (z) = z+z^2+z^3 +\dots = \frac{z}{1-z}.
$$
Such examples suggest the following claim.

\begin{thesis}
\label{thesys_subexp}
The ordinary
generating series $G_\p$ of a generic symmetric or non-symmetric operad $\p$ is rational
if the sequence $\{ \dim \p_n\}$ is
bounded by a polynomial in $n$.
\end{thesis}

In particular, if the sequence $\{ \dim \p_n\}$ is bounded by a
constant, then it is eventually periodic.

 To support our claims, we
consider operads with  finite Gr\"obner bases of relations.
The analogy with the theory of associative algebras leads us to expect
that the generating series of operads with finite Gr\"obner bases
are in our class.
This expectation is suggested by the results of this article.

However, the class of  operads which are generic in the sense of our Expectations 
do not cover all quotients of the associative operad. We discuss this later in Subsection~\ref{subs:ass}.

\subsection{Operads with quadratic Gr\"obner bases}

\label{subs:conj_PBW}

Let us formulate a couple of conjectures about PBW operads (i.e. the operads which admits a quadratic Gr\"obner basis).

\begin{conj}\label{conj::PBW::series}
The exponential generating series of a symmetric PBW operad  with binary generators is differential algebraic.
\end{conj}

This conjecture is motivated by the following stronger conjecture that has been checked in all known examples.

\begin{conj}\label{conj::PBW}
 Let $\p$ be a symmetric operad generated by binary operations and with quadratic relations that
form a Gr\"obner basis according to one of the admissible
orderings of shuffle monomials. Then there exists a monomial
shuffle regular operad $\q$ with the same  dimensions of
operations ($\dim\p(n)=\dim\q(n)$ for all $n$).
\end{conj}

Why might this conjecture be true? The most important assumption
is that the operad $\p$ is symmetric. Therefore, one should use
the representation theory of the permutation groups $\Sigma_n$.
First, using an upper-triangular change of basis one can choose
the set of generators $\Upsilon$ of an operad $\p$ such that the
transposition element $(12)\in \Sigma_2$ preserves this set. This
means that $\forall \alpha\in \Upsilon$ $(12)\cdot \alpha =
(-1)^{\epsilon} \alpha'$ where $\alpha'\in \Upsilon$. Suppose, for
simplicity, that $(12)\cdot \alpha = \pm \alpha$ for all
$\alpha\in \Upsilon$. 
For any given pair of generators $\alpha,\alpha'$ the set of all operations of the form
$\alpha\circ \alpha'$ form a 3-dimensional representation of
$\Sigma_3$ which is  induced from a one-dimensional
$\Sigma_2$-representation. Therefore, it decomposes into the sum
of one-dimensional and two-dimensional irreducible representation.
 The relations for the two-dimensional
representation are
$$
\alpha(\alpha'(x_1,x_2),x_3) + \alpha(\alpha'(x_2,x_3),x_1) + \alpha(\alpha'(x_3,x_1),x_2)
$$
The relations for the one-dimensional representation are
$$
\pm\alpha(\alpha'(x_1,x_2),x_3) =\pm \alpha(\alpha'(x_2,x_3),x_1) =\pm \alpha(\alpha'(x_3,x_1),x_2)
$$
On the other hand, there are two shuffle regular sets of quadratic
monomials:
$$ \{\alpha(\alpha'(x_1,x_2),x_3),\alpha(\alpha'(x_1,x_3),x_2)\}
\text{ and } \{\alpha(x_1,\alpha'(x_2,x_3))\} $$
We conjecture that 
the monomial operad $\q$ can be constructed by taking the leading
term $\alpha(x_1,\alpha'(x_2,x_3))$ of the relation that
nontrivially projects onto the corresponding one-dimensional
$\Sigma_3$-irreducible subspace and the set
$\{\alpha(\alpha'(x_1,x_2),x_3),\alpha(\alpha'(x_1,x_3),x_2)\}$
for those relations which project onto the two-dimensional
$\Sigma_3$-irreducible subspace.

\subsection{Varieties of associative algebras}

\label{subs:ass}
\label{subs:conj_ass}

The symmetric quotient operads of the operad $\ass$ (i.e., the
operads corresponding to the varieties of associative algebras
with polynomial identities; called here {\em associative operads})
form, probably, the widest extensively classified and studied
class of operads. Note that the traditional terminology for these
operads differs from the one we use. In particular, the dimensions
$\dim \p_n$  of an associative operad $\p$ are usually referred to
as {\em codimensions} of the corresponding variety of algebras
with polynomial identities; the exponential and the usual
generating series of the operad are referred as {\em the
codimension series} and {\em the exponential codimension series}
of the corresponding variety. (See, e.g.,~\cite{giza,Kanel_Regev}.)

 We do not know whether
each symmetric quotient operad of $\ass$ has a finite Gr\"obner
basis. On the other hand, one can construct examples of {\em
shuffle} monomial quotients of $\ass$ which are infinitely
presented, that is, which have no finite Gr\"obner basis, at least
with the standard choice of generators. However, the following
properties of the generating series of associative operads are close
to the properties of our operads with  finite shuffle regular
Gr\"obner bases.

Firstly, in many cases the dimensions  $\dim {\p}_n$  for such an
operad $\p$ can be calculated explicitly (see~\cite{ufn,
giza,dr, bd}). In such cases, these dimensions are presented as linear
combinations of integer exponents of $n$ with polynomial
coefficients. In particular, the corresponding exponential generating
functions  are differential algebraic, that
is, they satisfy  the conclusion of our
Corollary~\ref{cor::main_cor}.

Secondly, the only non-symmetric
associative operads are the finite-dimensional ones defined by the
identity $x_1x_2 \dots x_n =0$ (operads of nilpotent associative
algebras); obviously, the generating series of such operads 
satisfy the conclusions of our Theorem~\ref{th-nonsym} and
Corollary~\ref{cor::nonsym_constant_growth}. 

Thirdly, the sequence
of dimensions of an associative operad has either polynomial or
exponential growth (by the results of Regev and Kemer,
see~\cite[Theorems~4.2.4 and~7.2.2]{giza}). This fact is similar to our
Corollary~\ref{cor::sym_constant_growth}.

After these observations, we think that 
Corollary~\ref{cor::main_cor}  (which suggests Expectation~\ref{main_thesys})
holds in a more
general situation, that is, for a wider class of symmetric operads
containing both operads with a finite shuffle regular Gr\"obner
basis and associative operads.

\begin{conj}
\label{conj::assoc_general} The exponential generating series of
any associative operad  is differential algebraic.\footnote{After this paper was submitted, Berele has proved this conjecture by showing that the sequence $\{\dim \p(n)\}$ is holonomic for each associative operad $\p$, see~\cite{berele}.} 
\end{conj}

Note that the two languages of universal algebra (based on  the 
varieties or on the operads) are equivalent provided that the characteristic of the  basic
field $\kk$ is zero. In this case, our
Conjecture~\ref{conj::assoc_general} is equivalent to the following:
{\em the exponential codimension
series of each variety of PI algebras is differential algebraic}.

A stronger version of Conjecture~\ref{conj::assoc_general} holds
for associative operads of polynomial growth, that is, their
ordinary generating series are rational~\cite[Ex.~13]{bd}.
However, for some natural associative operads (e.g., for the
operad  defined by the identities of $3\times 3$ matrices) the
ordinary generating series are not rational~\cite{br}. This means
that the conclusion of
Corollary~\ref{cor::sym_constant_growth} does not hold for
general associative operads. In particular, this means that some
associative operads have no finite shuffle regular Gr\"obner basis
(whereas they are always finitely presented by Kemer's famous
solution of the Specht problem~\cite{kemer_Specht}). This means
also that some associative operads are not generic in the sense of
our Expectation~\ref{thesys_subexp}.

\subsection{Detailed description of generating series}

\label{subs:conj_genseries}

It would be interesting to find a smaller class of formal power
series than the class of differential algebraic ones which include
all (exponential) generating series of generic symmetric operads.
For the operads with a finite shuffle regular Gr\"obner basis, we
propose the following approach.

Let us define a binary  operation on $\QQ [[z]]$ by $f*g:= C(f,g)$,
       where $C(f,g)(z) = \int_0^z f'(w) g(w) \, dw$ is as defined in~(\ref{eq::C(f,g)}).
       Recall that the multiple operation $C$ from equation~(\ref{eq::C(1..n)})
       is also defined via this binary operation:
       $
       C(f_1, \dots , f_n) = (f_1 * (f_2 *\dots (f_{n-1} *f_{n})\dots )$.
       The algebra $\QQ [[z]]$ with respect to the operation $f*g$ satisfies the identity
       $$
        (a*b)*c = a*(b*c+c*b),
       $$
       that is, this is a Zinbiel algebra\footnote{We are grateful to Askar Dzhumadil'daev for pointing us this fact.}.
The Zinbiel operad  (also known as dual Leibniz operad, or the operad of chronological algebras)
were first introduced in~\cite{Loday_Zinbiel}.
       Then one can consider the main system~(\ref{eq::sym_system_eq}) as a set of
       algebraic equations over the Zinbiel subalgebra $\QQ [z] \subset \QQ [[z]]$.

In this connection, the following questions arise.

1. Does it follows that the exponential series $E_\p(z)$ satisfy a {\em single}
algebraic equation over $\QQ [z]$ with respect to the operation $'*'$, that is, does
it follows that $E_\p(z)$ is Zinbiel algebraic over $(\QQ [z],*)$?

2. What are the conditions on the (non-negative rational)
coefficients for a formal power series
to be Zinbiel algebraic (or for a collection $y_0(z),\ldots,y_N(z)$ of
 such formal power series to satisfy a system of Zinbiel algebraic equations)?

The answer to the second question could give a detailed
description of generating series of, at least, operads with
shuffle regular Gr\"obner basis.

\end{document}